\documentclass[10pt]{amsart}

\usepackage[english]{babel}
\usepackage{amsmath,amssymb,amsfonts,amsthm,amsopn}
\usepackage{latexsym,graphicx}
\usepackage[title]{appendix}
\usepackage{tikz-cd}
\usepackage{tikz}
\usepackage{mathtools}
\usepackage{multirow}
\usepackage{enumerate}
\usepackage[autostyle=true]{csquotes}
\mathtoolsset{showonlyrefs}

\newtheorem{theorem}{Theorem}[section]

\newtheorem{corollary}[theorem]{Corollary}
\newtheorem{proposition}[theorem]{Proposition}
\theoremstyle{definition}
\newtheorem{definition}[theorem]{Definition}
\newtheorem{example}[theorem]{Example}
\theoremstyle{remark}
\newtheorem{remark}[theorem]{Remark}

\newcommand{\field}[1]{\mathbb{#1}}
\newcommand{\bR}{\field{R}}
\newcommand{\bN}{\field{N}}
\newcommand{\bZ}{\field{Z}}
\newcommand{\bC}{\field{C}}

\newcommand{\cF}{\mathcal{F}}
\newcommand{\cS}{\mathcal{S}}
\newcommand{\cD}{\mathcal{D}}

\newcommand{\cE}{\mathcal{E}}
\newcommand{\cG}{\mathcal{G}}
\newcommand{\cM}{\mathcal{M}}

\newcommand{\cA}{\mathcal{A}}

\newcommand{\cI}{\mathcal{I}}

\def\rd{\bR^d}
\def\rdd{{\bR^{2d}}}

\def\zd{\bZ^d}

\def\ud{\,\mathrm{d}}

\def\Op{\operatorname{Op}}
\def\Opw{\operatorname{Op}_{\mathrm{w}}}

\DeclareMathOperator*{\supp}{supp}
\DeclareMathOperator*{\Sp}{Sp}
\DeclareMathOperator*{\Mp}{Mp}
\DeclareMathOperator*{\Sym}{Sym}

\DeclareMathOperator*{\GL}{GL}

\DeclareMathOperator*{\indlim}{ind\,lim}

\newcommand{\norm}[1]{\lVert#1\rVert}
\newcommand{\scal}[2]{\left\langle #1,#2\right\rangle}

\newcommand{\asympconst}{\asymp}

\begin{document}

 \setlength{\emergencystretch}{2em}

\title[Modulation spaces: an excursus]{Modulation spaces: from Feichtinger's original definition to modern Gabor and symplectic characterizations}

\author{Antonio Caputo}
\address{Department of Mathematics, University of Torino, Italy}
\email{antonio.caputo@unito.it}
\author{Elena Cordero}
\address{Department of Mathematics, University of Torino, Italy}
\email{elena.cordero@unito.it}
\author{Gianluca Giacchi}
\address{Dipartimento di Matematica, Universit\`a della Svizzera Italiana, Switzerland}
\email{gianluca.giacchi2@unibo.it}
\author{Luigi Rodino}
\address{Dipartimento di Matematica, University of Torino, Italy}
\email{luigi.rodino@unito.it}

\thanks{This chapter is dedicated to Hans G. Feichtinger on the occasion of his 75th birthday.}

\begin{abstract}

We present an expository excursus on modulation spaces, from Feichtinger's original construction on locally compact abelian groups to recent symplectic formulations. We trace the same local-global principle through bounded uniform partitions of unity, the short-time Fourier transform, coorbit methods, and Gabor coefficients, while distinguishing equivalent descriptions from the weighted, quasi-Banach, and Gelfand-Shilov extensions of the underlying functional setting. We also discuss Gabor matrices, convolution and embedding estimates, and an application to a nonlinear dispersive equation. The final part is devoted to metaplectic Wigner distributions $W_{\mathcal A}$. In the shift-invertible case these representations are, up to chirps and a linear change of phase-space variables, rescaled short-time Fourier transforms. This structural fact explains both their characterization of modulation and Wiener amalgam spaces and their discretization by metaplectic Gabor frames, and places the recent symplectic theory naturally within Feichtinger's modulation-space framework.

\end{abstract}

\maketitle

\section{Introduction}\label{sec:introduction}
Modulation spaces are function spaces defined by Hans Feichtinger, in which distributions are classified according to their joint concentration in position and frequency. They quantify these two aspects within a common phase-space framework, while allowing their relative roles to be distinguished by mixed norms and weights. Their construction grew out of Feichtinger's work on abstract harmonic analysis, Banach convolution algebras, Wiener amalgam spaces, and numerical questions, together with the conviction that a useful theory should connect continuous and discrete models, ordinary functions and generalized functions, and mathematical rigor with the needs of signal analysis.
This perspective, later described as part of  \emph{conceptual harmonic analysis}, is especially visible in Feichtinger's historical accounts of the subject \cite{Feichtinger2006LookingBack,feichtinger2024modelling}.

\subsection{The original uniform-decomposition perspective}
The original construction took shape at the beginning of the 1980s. Inspired by the decomposition methods used for Besov and Triebel-Lizorkin spaces, Feichtinger replaced their dyadic frequency coverings by uniform coverings and bounded uniform partitions of unity (BUPUs).
Local Fourier information was measured in a Fourier-Lebesgue norm and the resulting local quantities were assembled globally by an $\ell^q$-norm.
In present notation, the Fourier transform is
\begin{equation}\label{intro.defFT}
    \widehat f(\xi)=\cF f(\xi)=\int_{\rd}f(y)e^{-2\pi i\xi\cdot y}d y, \qquad \xi\in\rd,
\end{equation}
and the Euclidean model may be summarized schematically, in terms of {\em Wiener amalgam spaces}, as
\begin{equation}\label{intro.original-viewpoint}
 M^{p,q}(\rd)=\cF^{-1}W(\cF L^p,\ell^q)(\rd).
\end{equation}
For $1\leq p,q\leq\infty$, a smooth BUPU $(\psi_k)_{k\in\zd}$ adapted to a uniform frequency covering gives
\begin{equation}\label{intro.uniform-norm}
 \norm{f}_{M^{p,q}}
 \asymp
 \left\|\left(\|\cF^{-1}(\psi_k\widehat f)\|_{L^p}\right)_{k\in\zd}\right\|_{\ell^q}.
\end{equation}
Thus the local norm is applied to each frequency-localized piece after inverse Fourier transformation; the outer sequence norm assembles these pieces.

We write $M^p=M^{p,p}$.
The construction was originally developed in the setting of locally compact abelian (LCA) groups and first appeared in the influential 1983 technical report \cite{Feichtinger1983}. An updated version was later published in \cite{Feichtinger1983ModulationSpaces}.

The endpoint $S_0:=M^1$, now called the \emph{Feichtinger algebra}, had already appeared as a Fourier-invariant Segal algebra \cite{Feichtinger1981SegalAlgebra}.
In Feichtinger's own retrospective account, the passage from BUPUs to continuous moving windows revealed the decisive role of the short-time Fourier transform and of the Heisenberg group.

\subsection{The STFT and the coorbit perspective}
The subsequent comparison with wavelet analysis led, with Gr\"ochenig, to coorbit theory \cite{FeichtingerGrochenig1988,FeichtingerGrochenig1989}. In its basic integrable setting, coorbit theory constructs function spaces by measuring the coefficients of a unitary group representation in suitable function spaces on the group. For time-frequency analysis, the relevant representation is the {\em Schr\"odinger representation} of the reduced Heisenberg group:
\begin{equation}
    \rho(x,\xi;\tau)f(y)=e^{i\tau}e^{-i\pi x\cdot\xi}\pi(x,\xi)f(y), \qquad (x,\xi;\tau)\in\rdd\times[0,2\pi),
\end{equation}
where for $x,\xi\in\rd$, translation and modulation are the unitary operators
\begin{equation}\label{intro.tf-shifts}
 T_xf(y)=f(y-x),\qquad M_\xi f(y)=e^{2\pi i\xi\cdot y}f(y),
\end{equation}
and the associated time-frequency shift is
\begin{equation}
 \pi(z)=\pi(x,\xi)=M_\xi T_x,\qquad z=(x,\xi)\in\rdd.
\end{equation}
The commutation rule
\begin{equation}
 M_\xi T_x=e^{2\pi i\xi\cdot x}T_xM_\xi
\end{equation}
shows that $z\mapsto\pi(z)$ is a projective representation of phase space; the additional central variable in $\rho$ turns it into a representation of the reduced Heisenberg group.  This noncommutativity is not a technical nuisance: it is the algebraic mechanism behind time-frequency analysis and Gabor expansions, beginning with Gabor's foundational work \cite{gabor1946theory}.

For a nonzero window $g\in L^2(\rd)$, the short-time Fourier transform (STFT) of $f\in L^2(\rd)$ is
\begin{equation}\label{intro.defSTFT}
 V_gf(x,\xi)
 =\langle f,\pi(x,\xi)g\rangle=\int_{\rd}f(t)\overline{g(t-x)}e^{-2\pi i\xi\cdot t} dt,
 \qquad x,\xi\in\rd.
\end{equation}
Here, $\langle \cdot,\cdot \rangle$ denotes the sesquilinear standard inner product of $L^2(\rd)$. Observe that the definition extends uniquely to $(f,g)\in\cS'(\rd)\times\cS(\rd)$ by interpreting $\langle \cdot,\cdot \rangle$ as the duality pairing $\cS'(\rd)\times\cS(\rd)$, antilinear in the second component.
For fixed time $x\in\rd$, $V_gf(x,\xi)$ measures how the frequency $\xi$ contributes to $f$ in a neighborhood of $x$. A reverse interpretation is allowed thanks to the {\em fundamental identity of time-frequency analysis},
\begin{equation}\label{intro.FITA}
    V_{\widehat g}\widehat f(x,\xi)=e^{-2\pi ix\cdot \xi}V_gf(-\xi,x).
\end{equation}
The same coefficient can therefore be read either as frequency information near a fixed position or, after Fourier transformation, as position information near a fixed frequency.

The STFT leads to the modern weighted formulation developed in Section~\ref{sec:stft-definition}. There, a moderate phase-space weight $m$ is incorporated into a mixed norm on $V_gf$, and different admissible windows define equivalent norms. With suitable windows, the same construction extends to $0<\min\{p,q\}<1$ and yields quasi-Banach modulation spaces. 
Precisely, let $m$ be a $v$-moderate weight. This means that $v$ is {\em submultiplicative}, i.e., $v(z+w)\leq v(z)v(w)$, and $m(z+w)\lesssim m(z)v(w)$. The reader may refer to \cite{G2007} as an exhaustive reference for the role of weights in time-frequency analysis.
For $0<p,q\leq\infty$, the weighted mixed-(quasi-)norm space $L_m^{p,q}(\rdd)$ is endowed with
 \begin{equation}
 \norm{F}_{L_m^{p,q}}
 =\left(\int_{\rd}\left(\int_{\rd}|F(x,\xi)|^p m(x,\xi)^p\,dx\right)^{q/p}d\xi\right)^{1/q},
\end{equation}
with the standard endpoint modifications. For the properties of mixed-norm spaces we refer to \cite[Section 2.2]{CorderoRodino2020}. Then, for a fixed window $g\in\cS(\rd)\setminus\{0\}$,
\begin{equation}\label{intro.defimoderna}
    \norm{f}_{M^{p,q}_m}:=\norm{V_gf}_{L^{p,q}_m}.
\end{equation}
We write $M^{p,q}$ when the weight is identically one.
On the diagonal, $p=1$ gives the {\em Segal algebra} $S_0(\rd)=M^1(\rd)$, initially defined as the Wiener amalgam space $W(\cF L^1,\ell^1)(\rd)$ of functions that are locally in $\cF L^1(\rd)$ and globally summable,
whereas $p=2$ gives $L^2(\rd)$, the space of {\em finite-energy signals}. Technically, the term {\em signal} is reserved for functions on $\bR$, but along the years it became more and more common to use it also for functions on $\rd$.
The fundamental identity \eqref{intro.FITA} shows that $\cF:M^p(\rd)\to M^p(\rd)$ is a topological automorphism for every $0<p\leq\infty$. With a Fourier-invariant normalized window, such as a Gaussian, the corresponding STFT norm is preserved. On mixed spaces $M^{p,q}$ with $p\neq q$ the situation is different, see \eqref{eq:fourier-modulation-wiener} below.
This modern STFT definition is not a competing definition but the continuous, representation-theoretic form of the original Wiener-amalgam idea.
Time-frequency shifts act continuously on weighted spaces and isometrically in the unweighted case.
Duality, interpolation, convolution and pointwise-product estimates, embeddings, and atomic decompositions are all compatible with this phase-space geometry.  These properties make modulation spaces a natural replacement for Lebesgue or Sobolev spaces whenever localization in both variables matters, and account for their use in Gabor analysis, pseudodifferential and Fourier integral operators, dispersive and Schr\"odinger equations, quantum mechanics, and signal processing \cite{CorderoRodino2020,Grochenig2001}.

The endpoint $M^1=S_0$ occupies a central position in Feichtinger's perspective.  It is a Fourier-invariant Banach algebra under both convolution and pointwise multiplication, it is stable under time-frequency shifts, and it is minimal among the natural Banach spaces with these invariance properties.  Its dual is identified with $M^\infty$, and one obtains the Banach Gelfand triple
\begin{equation}\label{intro.gelfand-triple}
 S_0(\rd)=M^1(\rd)\hookrightarrow L^2(\rd)=M^2(\rd)
 \hookrightarrow M^\infty(\rd)=S_0'(\rd).
\end{equation}
This triple is Feichtinger's preferred arena for signal analysis: $S_0$ provides test functions that are simultaneously well localized in time and frequency, $L^2$ carries the unitary theory, and $S_0'$ contains the generalized signals that can be tested against every element of $S_0$.  In the terminology emphasized in his recent work, the elements of $S_0'$ are \emph{mild distributions}. Equivalently, they have bounded STFT with respect to one, hence every, nonzero $S_0$ window.  Dirac masses, Dirac combs, and translation-bounded measures fit naturally into this class. A sequential realization, inspired by Lighthill's approach to tempered distributions, identifies $S_0'(\rd)$ in a natural way, and with equivalent norms, with the space of equivalence classes of mild Cauchy sequences of bounded continuous functions, thus recovering the same Banach Gelfand triple
\cite{feichtinger2020sequential}. Unlike a framework based exclusively on pointwise functions or on the finer topology of the Schwartz class, the triple \eqref{intro.gelfand-triple} is well suited to sampling, kernels, weak-$\ast$ limits, and numerical approximation while retaining Fourier invariance \cite{MR2477142,feichtinger2024modelling,Jakobsen2018FeichtingerAlgebra}.

Parallel to time-frequency analysis runs multiresolution analysis, which analyzes signals in terms of position and resolution by replacing modulation with dilations. Its main tool is the so-called {\em wavelet transform}, constructed from a representation of the affine group instead of the Heisenberg group. Coorbit theory unifies time-frequency and multiresolution analysis within a common representation-theoretic framework where the analyzing transform is tied to the chosen representation.
Adopting an alternative point of view, the final part of this chapter introduces metaplectic time-frequency representations and examines their use in characterizing modulation-space (quasi-)norms. Rather than starting from the STFT as the prescribed analyzing transform, we consider a broader family of time-frequency representations and ask which of them characterize the same spaces $M_m^{p,q}$. 
Despite their different starting points, both approaches highlight the central role of the STFT in the characterization of modulation spaces. We shall see how this role emerges in the metaplectic setting in the final part of the chapter.

\subsection{Metaplectic operators}
In this work, we shall use metaplectic operators on $L^2(\rdd)$. The metaplectic group $\Mp(2d,\bR)$ is the two-fold cover of the symplectic group $\Sp(2d,\bR)$. The latter consists of $4d\times4d$ real matrices $\cA$ so that $\cA^\top J\cA=J$, where
\begin{equation}
    J=\begin{pmatrix}
        O_{2d} & I_{2d}\\
        -I_{2d} & O_{2d}
    \end{pmatrix}
\end{equation}
represents the canonical symplectic form of $\bR^{4d}$. Metaplectic operators are unitary on $L^2(\rdd)$ and they restrict to homeomorphisms of $\cS(\rdd)$. On top of this, they extend by duality to homeomorphisms of $\cS'(\rdd)$. If $ {\hat\cA}\in\Mp(2d,\bR)$, its {\em projection} $\pi^{Mp}(\hat \cA)=\cA\in\Sp(2d,\bR)$ is the unique symplectic matrix such that  {it }intertwines the Schr\"odinger representation of the Heisenberg group:
\begin{equation}
    \hat\cA\rho(z;\tau)\hat\cA^{-1}=\rho(\cA z;\tau), \qquad z\in\bR^{4d},\;\tau\in\bR.
\end{equation}
We shall partition $\cA$ into $d\times d$ blocks as
\begin{equation}\label{intro.blockA}
    \cA=\begin{pmatrix}
        A_{11} & A_{12} & A_{13} & A_{14}\\
        A_{21} & A_{22} & A_{23} & A_{24}\\
        A_{31} & A_{32} & A_{33} & A_{34}\\
        A_{41} & A_{42} & A_{43} & A_{44}
    \end{pmatrix}, \qquad A_{jk}\in\bR^{d\times d},\;j,k=1,\ldots,4.
\end{equation}
Conversely, $\cA$ determines $\hat\cA$ up to a sign, that is $\pi^{Mp}$ is a group homomorphism with kernel $\{\pm\mathrm{id}_{L^2}\}$.
For the purpose of this chapter, this sign is completely irrelevant, and will therefore be omitted. 
In particular, the analytic and geometric properties of the operator $\hat\cA$ can be related directly to the structure of a linear symplectic transformation, and read from it.
 {Standard metaplectic operators include} the unitarily-rescaled Fourier transform $i^{-d}\cF$,  {the} unitary rescalings
\begin{align}
    \mathfrak T_EF(z)=i^m|\det(E)|^{1/2}F(Ez), \qquad E\in\mathrm{GL}(2d,\bR),\; z\in\rdd,
\end{align}
where $m$ is  {chosen} properly according to $\mathrm{arg}\det(E)^{1/2}$, and chirp-products
\begin{align}
    \mathfrak p_QF(z)=\Phi_Q(z)F(z), \qquad Q\in\mathrm{Sym}(2d,\bR),
\end{align}
where $\Phi_Q(z)=e^{i\pi Qz\cdot z}$. In the analysis of time-frequency representations there is another metaplectic  {operator} that plays a fundamental role, the {\em partial Fourier transform with respect to the  {frequency} variables}:
\begin{equation}\label{intro.defF2}
    \cF_2F(x,\xi)=i^{-d/2}\int_{\rd}F(x,y)e^{-2\pi i\xi\cdot y}\mathrm dy, \qquad F\in\cS(\rdd), \;x,\xi\in\rd.
\end{equation}
As with the sign ambiguity in the choice of a metaplectic lift, we shall suppress the unimodular constants (\emph{phase factors}) appearing in the formulas above, since they do not affect the norm estimates and characterizations considered here. For simplicity, we therefore regard these operators up to an overall phase and identify them with their symplectic projections. 
This identification is understood at the projective level: the map $\pi^{Mp}$ itself remains a two-to-one group homomorphism, not an isomorphism. The phase factors are essential to the precise group structure of $\Mp(2d,\bR)$ and must be retained whenever exact operator identities are required.

For notation purposes, we defined metaplectic operators on $L^2(\rdd)$. We stress that the metaplectic group $\Mp(d,\bR)$ of operators acting on $L^2(\rd)$ is defined analogously, the projection $\pi^{Mp}:\hat\delta\in\Mp(d,\bR)\mapsto\delta\in\Sp(d,\bR)$ being a 2:1 group epimorphism onto the group of $2d\times2d$ symplectic matrices, herein denoted by $\Sp(d,\bR)$.

\subsection{Metaplectic Wigner distributions}
Time-frequency  {representations} defined using these operators date back to a paper of Bai, Li and Cheng, who defined the {\em new Wigner-Ville distribution} in \cite{Bai2012}, later generalized by Zhang and Luo in \cite{ZhangNewWigner}. In parallel (cross-){\em matrix Wigner distributions} were considered in \cite{Bayer2020,CordTrap} outside a metaplectic framework.
The widest possible generalization of these representations was first stated by two of the authors in \cite{CRPartI}.
\begin{definition}
    Let $\hat\cA\in\Mp(2d,\bR)$. The corresponding {\em metaplectic Wigner distribution}, or $\cA$-Wigner distribution, is the sesquilinear time-frequency representation
    \begin{equation}
        W_\cA(f,g)=\hat\cA(f\otimes\bar g), \qquad f,g\in L^2(\rd).
    \end{equation}
\end{definition}
Observe that the identification up to a phase $\cA\leftrightarrow\hat\cA$ is implied both in the notation $W_\cA$ and in the terminology ``$\cA$-Wigner distribution''. 
The properties of these objects have been established in the past years by the authors in \cite{CGRPartII,CorderoGiacchi2023JMPA,CorderoGiacchi2024MetaplecticGabor,CGRUnified}. More recently, a characterization in terms of the Schr\"odinger representation of the Heisenberg group was obtained in \cite{GS2026}. 

The main question we review in this chapter is 
\begin{quote}
    {\it Which metaplectic Wigner distributions can be used to measure the local time-frequency concentration of signals?}
\end{quote}
That is to say, under which conditions
\begin{equation}\label{intro.equiv.mod.spaces}
    \norm{f}_{M^{p,q}_m}\asymp \norm{W_\cA(f,g)}_{L^{p,q}_m},
\end{equation}
where $g\in\cS(\rd)\setminus\{0\}$ is a fixed window, is true for every $0<p,q\leq\infty$ and moderate weight $m$.

This question was first answered in \cite[Corollary 3.6]{CRPartI} for {\em $\tau$-Wigner distributions} ($\tau\in\bR$)
\begin{equation}\label{intro.def.Wtau}
    W_\tau(f,g)(x,\xi)=\int_{\rd}f(x+\tau y)\overline{g(x-(1-\tau)y)}e^{-2\pi i\xi\cdot y}\ud y, \qquad f,g\in L^2(\rd),\; x,\xi\in\rd,
\end{equation}
thereby extending to $M^{p,q}_m$, de Gosson's result \cite[Proposition 397]{DeGosson2011}, stated for the classical {\em Wigner distribution} ($\tau=1/2$ in \eqref{intro.def.Wtau}) and $M^p_{v_s}$ spaces, here $v_s(z)=(1+|z|^2)^{s/2}$, $s\in\bR$.

Back to metaplectic Wigner  {distributions} in their full generality, in \cite{CorderoRodino2023JFA} the following pivotal condition stemmed naturally while proving the inequality
\begin{equation}
    \norm{f}_{M^{p}_{v_s}}\lesssim \norm{W_\cA(f,g)}_{L^{p}_{v_s}}.
\end{equation}
\begin{definition}\label{intro.defSI}
    A metaplectic Wigner distribution is {\em shift-invertible} if there exists a matrix $E_\cA\in\bR^{2d\times2d}$ such that
    \begin{equation}\label{intro.defEAtrnasls}
        |W_\cA(\pi(z)f,g)|=|T_{E_\cA z}W_\cA(f,g)|, 
    \end{equation}
    holds for every $z\in\rdd$ and every $f,g\in L^2(\rd)$.
\end{definition}
A simple computation shows that if \eqref{intro.defEAtrnasls} holds and $\cA$ has blocks as in \eqref{intro.blockA}, then it must be
\begin{equation}\label{intro.defEA}
    E_\cA=\begin{pmatrix}
        A_{11} & A_{13}\\
        A_{21} & A_{23}
    \end{pmatrix},
\end{equation}
see \cite[Equation (73)]{CorderoRodino2023JFA}. In particular,
\begin{quote}
    {\it $W_\cA$ is shift-invertible if and only if $E_\cA\in \mathrm{GL}(2d,\bR)$}.
\end{quote}
In \cite[Theorem 2.28]{CorderoRodino2023JFA}, the authors prove that {\em covariant} shift-invertible metaplectic Wigner distributions satisfy \eqref{intro.equiv.mod.spaces} for $m=v_s$. We recall that $W_\cA$ is covariant if
\begin{equation}
    W_\cA(\pi(z)f,\pi(z)g)=W_\cA(f,g)(\cdot-z), \qquad f,g\in L^2(\rd)
\end{equation}
for every $z\in\rdd$.  {Their proof relies} on the characterization of covariant and shift-invertible metaplectic Wigner distributions, given either as an integral transform \cite[Theorem 2.27]{CorderoRodino2023JFA} or in terms of the corresponding symplectic projection \cite[Proposition 2.25]{CorderoRodino2023JFA}.

In \cite{CGRPartII}, {\em Wigner decomposable} metaplectic Wigner distributions are considered. These are distributions in the form
\begin{equation}
    W_\cA(f,g)=\mathfrak p_Q\cF_2\mathfrak T_E(f\otimes \bar g), \qquad f,g\in L^2(\rd),
\end{equation}
for $Q\in\Sym(2d,\bR)$ and $E\in\mathrm{GL}(2d,\bR)$. These are essentially matrix Wigner distributions. In \cite[Section 4]{CGRPartII}, a condition under which these time-frequency representations are shift-invertible is obtained, together with the corresponding characterization of modulation spaces $M^{p,q}_m$.

So far, the characterization of modulation spaces has been only given for particular sub-classes of metaplectic Wigner  {distributions} and shift-invertibility appeared as a background condition that for some reason is always in the way. It is only in \cite{CorderoGiacchi2023JMPA} that shift-invertibility is established as the key property for a metaplectic time-frequency representation to measure local time-frequency concentration, where the  {full characterization} of \eqref{intro.equiv.mod.spaces} is provided in the Banach setting. We remark that when $p\neq q$, shift-invertibility alone is not sufficient to obtain \eqref{intro.equiv.mod.spaces}, but a further condition on $E_\cA$ is needed, consistently with the results in \cite{FuhrShafkulovska2024} about the boundedness of linear change of variables  {within} mixed-norm Lebesgue spaces. Specifically, if $p\neq q$, the further condition is $A_{21}=O_d$. On top of this, the weight has to satisfy $m\asymp m\circ E_\cA {,}$ a condition that, for polynomial weights $v_s$, is trivially true.
The result at issue is \cite[Theorem 3.7]{CorderoGiacchi2023JMPA} and its proof relies on direct estimates. 
In the following work \cite[Theorem 7.1]{CorderoGiacchi2024MetaplecticGabor}, the result is extended to the whole quasi-Banach setting by means of the explicit characterization of shift-invertible metaplectic Wigner distributions in Corollary 4.4 therein. Importantly, the authors show that a $W_\cA$ is shift-invertible if and only if it is a rescaled STFT, up to rescalings and chirp-products. We refer to Theorem~\ref{thm:rescaled-stft} below for the details. 
However, the finest characterization of shift-invertible metaplectic Wigner distributions is given in \cite[Theorem 3.4]{CorderoGiacchi2024}, where it is proven that the mapping
\begin{equation}
    (E,Q,\widehat\delta)\in\mathrm{GL}(2d,\bR)\times\Sym(2d,\bR)\times\Mp(d,\bR)\longmapsto W_\cA(f,g)=\mathfrak T_E\mathfrak p_QV_{\hat\delta g}f
\end{equation}
is a bijection (observe that there is no phase defining $W_\cA$).

Moreover, in \cite{CorderoGiacchi2024MetaplecticGabor}, the authors consider {\em metaplectic Gabor frames}, and obtain the atomic characterization of modulation spaces through shift-invertible metaplectic Wigner distributions, see \cite[Theorem 7.3]{CorderoGiacchi2024MetaplecticGabor}. We refer to Section \ref{sec:gabor} for the detailed definition of Gabor frame.
It is  {worth stressing} that similar results are obtained for Wiener amalgam spaces in \cite[Theorem 7.2]{CorderoGiacchi2024MetaplecticGabor}. In this case, the condition $A_{21}=O_d$ for $E_\cA$ is replaced by $A_{13}=O_d$.
Up to the underlying structural conditions on $E_\cA$ and $m$, shift-invertibility is therefore sufficient for a metaplectic Wigner distribution to characterize quasi-Banach modulation spaces. The loop is closed by \cite[Theorem 1.6]{Giacchi}, where it is proven that in order for $W_\cA$ to characterize $M^p$ spaces, shift-invertibility is also necessary. The density of matrices $\cA\in\Sp(2d,\bR)$ satisfying the shift-invertibility condition $E_\cA\in\mathrm{GL}(2d,\bR)$ in $\Sp(2d,\bR)$ is proven in the same work, see \cite[Theorem 1.5]{Giacchi}.

\medskip
By any measure, the STFT is therefore the undisputed protagonist of symplectic analysis of time-frequency spaces. Its central role emerges from rather different principles in coorbit theory and in the metaplectic setting.
Understanding this connection is one of the themes of the present chapter, to which we shall return in the final part.
In conclusion, the recurring appearance of the STFT illustrates the remarkable lasting strength of Feichtinger's original time-frequency viewpoint.

\medskip
\noindent\textbf{Outline.}
Section~\ref{sec:feichtinger-original} recalls the original definition in the setting of locally compact abelian groups and compares its uniform frequency geometry with the dyadic geometry of Besov spaces. Section~\ref{sec:stft-definition} passes to the Euclidean STFT formulation and the general weighted theory. Section~\ref{sec:gabor} is devoted to Gabor frames and discrete modulation-space norms, Section~\ref{sec:gabor-operators} to Gabor matrices of operators and the Sj\"ostrand class, and Section~\ref{sec:quasi-banach} to the extension of these ideas to the quasi-Banach regime. Section~\ref{sec:weights-ultra} treats faster weights, Gelfand-Shilov spaces, and ultradistributions. Section~\ref{sec:guo-baoxiang} returns to frequency-uniform decompositions and illustrates their usefulness through a representative nonlinear PDE application. Finally, Section~\ref{sec:symplectic} studies metaplectic Wigner distributions, their modulation and Wiener amalgam characterizations, a brief comparison with the Cohen class, and their discretization by metaplectic Gabor frames.

\section{Feichtinger's original construction}
\label{sec:feichtinger-original}

\subsection{The locally compact abelian setting}

Let $G$ be a LCA group and $\widehat{G}$ its dual group.  {As recalled in the Introduction, Feichtinger's original construction was formulated in this general harmonic-analysis setting; here we spell out the LCA ingredients needed below.}

Characters play a central role in the theory of LCA groups, as they take the place of complex exponentials in the classical Euclidean setting, and are the key objects used to define the Fourier transform in the general setting. Throughout this section we write the group operation on $G$ additively.

\begin{definition}
Let $G$ be an LCA group. A (continuous) character on $G$ is a continuous group homomorphism $\xi:G\to\mathbb{S}^1$, where $\mathbb{S}^1:=\{e^{2\pi i\theta} \mid \theta\in\bR\}\subseteq\bC$ is the unit circle group (torus) in $\bC$, i.e., $\xi$ satisfies $\xi(x+y)=\xi(x)\xi(y)$ and $|\xi(x)|=1$ for all $x,y\in G$.
\end{definition}

The dual group $\widehat{G}$ of $G$ is the set of all characters on $G$, and is a group under the operation of pointwise multiplication $(\xi+\eta)(x):=\xi(x)\eta(x)$ for all $x\in G$. If we endow $\widehat{G}$ with the compact-open topology, the topology of uniform convergence on compact sets, $\widehat{G}$ becomes an LCA group.

Translation and modulation operators are defined according to the algebraic structure of the group $G$ as follows:
\begin{equation}
T_uf(x):=f(x-u), \qquad M_\xi f(x):=\xi(x)f(x), \qquad x,u\in G, \quad \xi\in\widehat{G}.
\end{equation}
In general, $T_u$ and $M_\xi$ do not commute, but enjoy the commutation relation $M_\xi T_u=\xi(u)T_uM_\xi$.

\begin{theorem}[Haar measure]\label{thm:haar-measure}
Every LCA group $G$ admits a nonzero translation-invariant regular Borel measure, finite on compact sets and unique up to a positive factor.
\end{theorem}
We write $dx$ for a fixed Haar measure on $G$ and choose the dual Haar measure on $\widehat G$ compatibly with Fourier inversion. Lebesgue measure on $\bR^d$, counting measure on $\bZ$, and normalized arc-length measure on $\mathbb S^1$ are the familiar examples. All $L^p$ spaces below refer to these measures.

\begin{definition}
Let $G$ be an LCA group, and $f\in L^1(G)$. The Fourier transform of $f$ is defined as follows:
\begin{equation}
\widehat{f}(\xi):=\int_G\overline{\xi(x)}f(x)dx, \qquad \xi\in\widehat{G}.
\end{equation}
\end{definition}

The basic $L^1$ estimate is unchanged:
\begin{equation}
 \|\widehat f\|_{L^\infty(\widehat G)}\leq\|f\|_{L^1(G)},
\end{equation}
since characters have modulus one.

The action of the Fourier transform interchanges the translation operator $T_u$ and the modulation operator $M_\xi$: for $u\in G$, $\xi,\eta\in\widehat{G}$ and $f\in L^1(G)$, it holds
\begin{equation}
\widehat{M_\eta f}(\xi)=T_\eta\widehat{f}(\xi), \qquad \widehat{T_uf}(\xi)=M_{-u}\widehat{f}(\xi),
\end{equation}
where, by Pontryagin duality, $u\in G$ is identified with the character $\xi\mapsto\xi(u)$ of $\widehat G$, so that $M_{-u}\widehat f(\xi)=\overline{\xi(u)}\,\widehat f(\xi)$.
The Riemann-Lebesgue theorem strengthens the conclusion to $\widehat f\in C_0(\widehat G)$: the transform is continuous and vanishes at infinity.

The Haar measure allows us to define the convolution product of $L^1$ functions as follows:
\begin{equation}
(f*g)(x):=\int_{G}f(y)g(x-y)dy, \qquad f,g\in L^1(G),
\end{equation}
which turns $L^1(G)$ into a commutative Banach algebra. The fundamental property
\begin{equation}
\widehat{f*g}=\widehat{f}\,\widehat{g}, \qquad f,g\in L^1(G),
\end{equation}
states that $\cF$ is an algebra homomorphism from $L^1(G)$ to $C_0(\widehat{G})$, considered with pointwise multiplication.

We will often consider {\em submultiplicative} weight functions on the group $G$, i.e., symmetric, locally bounded and measurable functions $w:G\to[1,+\infty)$ such that $w(x+y)\leq w(x)w(y)$ for all $x,y\in G$, and the weighted $L^p$ spaces $L_w^p(G)$ of functions $f$ such that $fw\in L^p$ endowed with norm $\|f\|_{L^p_w}=\|fw\|_{L^p}$. In particular, $L_w^1(G)$ is called Beurling algebra, and the space $A_w(\widehat{G}):=\cF L_w^1(G)$ is a Banach algebra of continuous functions on $\widehat{G}$ under pointwise multiplication.

More generally, we can consider for a submultiplicative weight function $w$, a $w$-{\em moderate} weight $m:G\to(0,+\infty)$, i.e., a function such that
\begin{equation}
m(x+y)\lesssim w(x)m(y), \quad \text{ for all } x,y\in G,
\end{equation}
and the weighted $L^p$ spaces $L_m^p(G)$ with respect to $m$.

In Feichtinger's original LCA framework the submultiplicative control weight is taken to satisfy the Beurling-Domar non-quasianalyticity condition
\begin{equation} \label{eq:Beurling-Domar}
\sum_{n=1}^{+\infty}n^{-2}\log w(nx)<\infty \quad \text{ for all } x \in G.
\end{equation}
For example, $w(x)=(1+|x|)^a$ with $a\geq0$ satisfies this condition on $\rd$. It is convenient to use the equivalent radial weights below. For $s\in\bR$ and any Euclidean dimension $n$, we use the single notation
\begin{equation}\label{eq:polynomial-weight}
 v_s(x)=\langle x\rangle^s=(1+|x|^2)^{s/2},\qquad x\in\bR^n.
\end{equation}
The ambient dimension of $v_s$ is always determined by context, we will primarily use $n=d$ or $n=2d$. Thus $1\otimes v_s$ and $v_s\otimes1$ denote one-sided weights on phase space, while $v_s$ used directly on a phase space denotes the radial polynomial weight there. See \cite{Feichtinger1983ModulationSpaces} for the original LCA assumptions.

In order to define Wiener-type spaces $W(B,C)$ for general Banach spaces $B$ and $C$, it is crucial to consider Banach spaces which behave well with respect to $A_w(\widehat{G})$. In particular, following the framework of \cite{Feichtinger1983}, we say that a Banach space $B$ is in {\em standard situation} with respect to $A_w(\widehat{G})$ if it satisfies the following three conditions:
\begin{enumerate}[1.]
\item $(A_w)_0\hookrightarrow B \hookrightarrow (A_w)_0'$, where $(A_w)_0:=A_w\cap C_c(\widehat{G})$ is considered as a topological vector space with respect to its natural inductive limit topology, $(A_w)_0'$ denotes the topological dual and $\hookrightarrow$ means continuous embedding.
\item $B$ is a Banach module over $A_w(\widehat{G})$ with respect to pointwise multiplication, i.e., $\|hf\|_B\lesssim\|h\|_{A_w}\|f\|_B$ for $h\in A_w(\widehat{G})$ and $f\in B$.
\item $B$ is a Banach module over $L_{w_{\mathrm{dual}}}^1(\widehat{G})$ with respect to convolution on $\widehat{G}$, where $w_{\mathrm{dual}}$ is a weight function on $\widehat{G}$ satisfying the Beurling-Domar condition \eqref{eq:Beurling-Domar}.
\end{enumerate}

Particularly relevant examples are the Fourier transforms of weighted, solid $BF$-spaces on $G$. Recall that a Banach space $B$ is called a $BF$-space on $G$ if it is continuously embedded into the space $L^1_{\text{loc}}(G)$ of locally integrable functions on $G$, endowed with the family of seminorms $\|f\chi_K\|_{L^1}$, where $K\subseteq G$ is compact and $\chi_K$ is the characteristic function of $K$. $B$ is said to be solid if
\begin{equation}
g\in L^1_{\text{loc}}(G), \ f\in B, \ |g(x)|\leq|f(x)| \text{ locally a.e.} \Rightarrow g\in B, \ \|g\|_B\leq\|f\|_B.
\end{equation}

Given a Banach space $B$ of distributions on the LCA group $\widehat{G}$ which is in standard situation with respect to $A_w(\widehat{G})$, and a continuous moderate weight $v$ on $\widehat G$, we can describe the Wiener-type space $W(B,L^q_v)(\widehat{G})$ as follows: fixing any test function $g\in A_w(\widehat{G})\cap C_c(\widehat{G})\setminus\{0\}$, we have
\begin{equation}
W(B,L^q_v)(\widehat{G})=\{f\in B_{\text{loc}} \mid F^{(g)}:t\mapsto\|(T_tg)f\|_B\in L^q_v(\widehat{G})\},
\end{equation}
and
\begin{equation}
\|f\|_{W(B,L^q_v)}:=\left(\int_{\widehat{G}}|F^{(g)}(t)|^qv(t)^q\,dt\right)^{1/q}, \qquad \text{ for } 1\leq q<\infty,
\end{equation}
with the usual modification for $q=\infty$.

Notice that $B_{\text{loc}}$ is the set of all distributions, i.e., elements of the dual of $A_w(\widehat{G})\cap C_c(\widehat{G})$ which belong to $B$ locally, and thus $(T_tg)f$ belongs to $B$ for all $t\in\widehat G$. This definition allows us to show that different test functions $g$ define the same space (for $B$, $q$, and $v$ fixed) yielding equivalent norms. Furthermore, there is an equivalent discrete characterization, via uniform partitions of unity, which we now recall.

Let us call a family $(\psi_i)_{i\in I}$ a bounded, uniform partition of unity in $A_w(\widehat{G})$ if there exists some relatively compact set $\widehat{Q}\subseteq\widehat{G}$ such that
\begin{enumerate}[1.]
    \item $\sum_{i\in I}\psi_i\equiv1$ and the family is locally finite.
 \item  $\sup_{i\in I}\|\psi_i\|_{A_w(\widehat{G})}<\infty$ (uniform boundedness).
    \item $\supp\psi_i\subseteq t_i+\widehat{Q}$ for $i\in I$.
    \item $\sup_{i\in I}|\{j \mid (t_i+\widehat{Q})\cap(t_j+\widehat{Q})\neq\emptyset\}|<\infty$ (uniform finite overlap).
\end{enumerate}
Then $f\in B_{\text{loc}}$ belongs to $W(B,L^q_v)$ if and only if
\begin{equation}
\|f\|_{D(\widehat{Q},B,\ell^q_v)}:=\left(\sum_{i\in I}\|f\psi_i\|_B^qv(t_i)^q\right)^{1/q}<\infty,
\end{equation}
with the usual modification for $q=\infty$, and this expression defines an equivalent norm on $W(B,L^q_v)$.

The role of the two control weights is different: $w$ governs the Fourier algebra used for localization, while $w_{\mathrm{dual}}$ controls translations on the dual group.

We can now assemble the construction of modulation spaces on $G$. If $B_0$ is an admissible solid BF-space on $G$ in the sense of \cite{Feichtinger1983ModulationSpaces}, its Fourier image is an admissible local component on $\widehat G$, and
\begin{equation}\label{eq:lca-modulation-amalgam}
 M(B_0,L_v^q)(G)=\cF_G^{-1}W(\cF_G B_0,L_v^q)(\widehat G),
 \qquad 1\leq q\leq\infty.
\end{equation}
Here $\|\widehat f\|_{\cF_G B_0}=\|f\|_{B_0}$. Thus modulation spaces are obtained by localizing on the dual group, measuring the inverse Fourier transforms of the localized pieces in $B_0$, and imposing the global $L_v^q$ condition. The next subsection makes this mechanism explicit when $B_0=L^p(\rd)$.

\subsection{Original definition by convolution with modulated windows}

In the Euclidean case $G=\bR^d$, the modulation operator reads
  \begin{equation}
 M_{\xi}f(x)=e^{2\pi i x\cdot \xi}f(x), \qquad x,\xi\in\rd.
\end{equation}
Let $g\in\cS(\rd)\setminus\{0\}$. Feichtinger's historical Euclidean notation $M_s^{p,q}$ corresponds, in the modern weighted notation used in the rest of this chapter, to the \emph{frequency-weighted} space

\begin{equation}
 M_s^{p,q}(\rd)=M_{1\otimes v_s}^{p,q}(\rd).
\end{equation}

Thus, for $1\leq p,q\leq\infty$,

\begin{equation}
\label{eq:feichtinger-original-norm}
 \norm{f}_{M_{1\otimes v_s}^{p,q}}
 =\left(\int_{\rd} \norm{M_{\xi}g*f}_{L^p}^q \langle \xi\rangle^{sq}\,d\xi\right)^{1/q},
\end{equation}

with the usual modification for $q=\infty$. See
\cite{Feichtinger1983,Feichtinger1983ModulationSpaces,Feichtinger1989Atomic}.  Throughout the chapter we display the frequency factor $1\otimes v_s$ explicitly, and reserve $m$ for a general phase-space weight.

  \begin{remark}[From the historical weight to a phase-space weight]
 Formula~\eqref{eq:feichtinger-original-norm} is a special case of the modern weighted theory.  If $m$ is a polynomially moderate weight on $\rdd$ and $g(t)=\overline{\varphi(-t)}$, then

\begin{equation}
 |(M_\xi g*f)(x)|=|V_\varphi f(x,\xi)|,
\end{equation}

so that, anticipating Section~\ref{sec:stft-definition},

\begin{equation}
 \norm{f}_{M_m^{p,q}}
 \asymp
 \left(\int_{\rd}\left(\int_{\rd}|(M_\xi g*f)(x)|^p m(x,\xi)^p\,dx\right)^{q/p}d\xi\right)^{1/q}.
\end{equation}

Hence the convolution formulation itself does not force a frequency-only weight. The latter is a historical specialization.  This is also consistent with the general LCA construction $M(B,L_v^q)(G)$ in Feichtinger's 1983 report
\cite{Feichtinger1983ModulationSpaces}.
\end{remark}

\subsection{Basic structural properties in the historical specialization}

The original theory already contains
 independence of the auxiliary window, completeness, duality, interpolation, convolution, trace and embedding results, together with continuous and discrete atomic descriptions.  To avoid confusing the special weight $1\otimes v_s$ with a general phase-space weight, we record here only the historical specialization and state the general weighted forms in Section~\ref{sec:stft-definition}.

\begin{theorem}[Independence of the window]
\label{thm:window-independence-original}
Let $1\leq p,q\leq\infty$ and $s\in\bR$. If $g_1,g_2\in\cS(\rd)\setminus\{0\}$ are admissible windows, then the norms defined by~\eqref{eq:feichtinger-original-norm} with $g=g_1$ and $g=g_2$ are equivalent.
\end{theorem}

 \begin{theorem}[Discrete characterization]
\label{thm:discrete}
Let $1\leq p,q\leq\infty$, $s\in\bR$, and let $(\psi_i)_{i\in I}$ be a bounded uniform partition of unity adapted to a uniform covering $(t_i+\widehat Q_0)_{i\in I}$ of $\rd$. Then
\begin{equation}
 \norm f_{M_{1\otimes v_s}^{p,q}}
 \asymp
 \left(\sum_{i\in I}\|\cF^{-1}(\psi_i\widehat f)\|_{L^p}^q\,\langle t_i\rangle^{sq}\right)^{1/q},
\end{equation}
with the usual modification for $q=\infty$.
\end{theorem}

 \begin{theorem}[Atomic characterization]
\label{thm:atomic}
For the same exponents, weight, and covering, $f\in M_{1\otimes v_s}^{p,q}(\rd)$ if and only if it admits a decomposition $f=\sum_{i\in I}f_i$
in $\cS'(\rd)$ with $\supp\widehat f_i\subseteq t_i+\widehat Q_0$ and
\begin{equation}
 \left(\sum_{i\in I}\|f_i\|_{L^p}^q\langle t_i\rangle^{sq}\right)^{1/q}<\infty.
\end{equation}
The infimum over all such decompositions gives an equivalent norm.
   This is the uniform-frequency analogue of atomic descriptions of Besov spaces. See
\cite{Feichtinger1989Atomic,Feichtinger1983ModulationSpaces}.
\end{theorem}

\subsection{Relation with Besov spaces}
The comparison with Besov spaces is useful because it isolates the geometric choice that distinguishes the two scales. Both are built by decomposing frequency space and aggregating local pieces, but Besov spaces use dyadic annuli whereas Feichtinger's construction uses uniformly sized frequency cells. For $s\in\bR$ and $1\leq p,q\leq\infty$, choose $\varphi_0,\varphi\in\cS(\rd)$ such that
\begin{equation}
 \supp\varphi_0\subseteq\{|\xi|\leq2\}, \quad \supp\varphi\subseteq\left\lbrace\frac{1}{2}\leq|\xi|\leq2\right\rbrace, \quad \varphi_0(\xi)+\sum_{j=1}^{\infty}\varphi(2^{-j}\xi)\equiv1.
\end{equation}
With $\varphi_j(\xi)=\varphi(2^{-j}\xi)$ for $j\geq1$, the inhomogeneous Besov norm is
\begin{equation}
 \|f\|_{B^{p,q}_s}:=\left(\sum_{j=0}^{\infty}2^{sjq}\|f*\cF^{-1}(\varphi_j)\|_{L^p}^q\right)^{1/q},
\end{equation}
with the usual modification for $q=\infty$. The parameter $s$ measures smoothness, $p$ measures the local $L^p$ size, and $q$ aggregates the dyadic pieces. In particular, $B_s^{2,2}=H^s$, while the scale also contains the classical H\"older-Zygmund spaces at the appropriate endpoint.

The formal similarities with modulation spaces are substantial. Both scales enjoy independence of the chosen decomposition, completeness, duality, interpolation, and precise embedding results. The essential difference is the covering geometry: modulation spaces resolve frequency space uniformly, while Besov spaces become coarser at high frequencies. The following standard facts make the comparison precise.

\begin{proposition}[Standard Besov properties]
\label{thm:window-independence-besov}\label{thm:completeness-besov}
\label{thm:duality-besov}\label{thm:interpolation-besov}
Let $1\leq p,q\leq\infty$ and $s\in\bR$.
\begin{enumerate}[(i)]
 \item Different admissible dyadic partitions yield equivalent norms on $B_s^{p,q}$. This is a Banach space with continuous embeddings $\cS\hookrightarrow B_s^{p,q}\hookrightarrow\cS'$. If $p,q<\infty$, then $\cS$ is dense and $(B_s^{p,q})^*=B_{-s}^{p',q'}$ under the distributional pairing.
 \item If $1\leq p_1,q_1<\infty$, $1\leq p_2,q_2\leq\infty$, $s_1,s_2\in\bR$, and $0<\theta<1$, then
 \begin{equation}
 (B_{s_1}^{p_1,q_1},B_{s_2}^{p_2,q_2})_{[\theta]}=B_s^{p,q},
 \end{equation}
 with equivalent norms, where
 \begin{equation}
 \frac1p=\frac{1-\theta}{p_1}+\frac\theta{p_2},\qquad
 \frac1q=\frac{1-\theta}{q_1}+\frac\theta{q_2},\qquad
 s=(1-\theta)s_1+\theta s_2.
 \end{equation}
\end{enumerate}
\end{proposition}

The following embeddings quantify the difference between the two covering geometries; the shifts in smoothness compensate for the number of uniform cells inside a dyadic annulus. See \cite{CorderoRodino2020} for this comparison.

\begin{theorem}
\label{thm:embedding-Besov-modulation}
Let $1\leq p_1,q_1,p_2,q_2,p,q\leq\infty$ be such that $p_1\leq p\leq p_2$ and $q_1\leq q\leq q_2$. Then we have the embeddings
\begin{equation}
B_{d\theta_1(p_1,q_1)}^{p_1,q_1}(\rd)\hookrightarrow M^{p,q}(\rd)\hookrightarrow B_{d\theta_2(p_2,q_2)}^{p_2,q_2}(\rd),
\end{equation}
where
\begin{equation}
\theta_1(p,q)=\max\left(0,\frac{1}{q}-\min\left(\frac{1}{p},\frac{1}{p'}\right)\right), \quad \theta_2(p,q)=\min\left(0,\frac{1}{q}-\max\left(\frac{1}{p},\frac{1}{p'}\right)\right).
\end{equation}
\end{theorem}

\section{The STFT formulation and the general weighted theory}
\label{sec:stft-definition}

The reference framework in this section is Chapter~2 of Cordero-Rodino
\cite{CorderoRodino2020}. Proofs of the standard statements below can be found there.  In order to remain inside $\cS'(\rd)$, we assume in this section that the control weights  {are moderate and have at most polynomial growth}.  Faster weights, which require Gelfand-Shilov test functions and ultradistributions, are treated in Section~\ref{sec:weights-ultra}.

\subsection{The short-time Fourier transform}
In view of its pivotal importance, let us describe briefly the main properties of the STFT.
 {For $f,g\in L^2(\rd)$, we use the definition in \eqref{intro.defSTFT} and proceed directly to the properties needed below.} 

For a fixed $g\in L^2(\rd)$, $V_gf$ is a uniformly continuous function in $L^2(\rdd)$, see \cite[Proposition 1.2.10 and Corollary 1.2.12]{CorderoRodino2020}, satisfying the {\em Moyal's identity},
\begin{equation}\label{moyalForSTFT}
    \langle V_{g_1}f_1,V_{g_2}f_2\rangle=\langle f_1,f_2\rangle\langle g_2,g_1\rangle, \qquad f_{1,2},g_{1,2}\in L^2(\rd),
\end{equation}
see \cite[Theorem 1.2.11]{CorderoRodino2020}. Moyal's identity yields an inversion formula for the STFT that reads as follows, see \cite[Theorem 1.2.16]{CorderoRodino2020} and the preliminary discussion therein.

\begin{theorem}\label{inversionSTFT}
    Let $g,\gamma\in L^2(\rd)$ be such that $\langle g,\gamma\rangle\neq0$. Then, for every $f\in L^2(\rd)$, 
    \begin{equation}\label{formulainversionSTFT}
        f=\frac{1}{\langle\gamma,g\rangle}\int_{\rdd}V_gf(z)\pi(z)\gamma\ud z,
    \end{equation}
    where the integral is interpreted in the distributional sense.
\end{theorem}
Another consequence of Moyal's formula is that, under the assumption of Theorem~\ref{inversionSTFT}, if $V_\gamma^\ast$ denotes the adjoint of $V_\gamma$, then
\begin{equation}
   \frac 1{\langle \gamma,g\rangle} V^\ast_\gamma V_g= \mathrm{id}_{L^2},
\end{equation}
as described in \cite[Section 1.2.4]{CorderoRodino2020}.
By \cite[Theorem 1.2.23]{CorderoRodino2020}, $(f,g)\in\cS(\rd)\times\cS(\rd)\mapsto V_gf\in\cS(\rdd)$.

 {As noted after \eqref{intro.defSTFT}, for $f\in\cS'(\rd)$ and $g\in\cS(\rd)$ the STFT is defined by duality and is a continuous function on $\rdd$; see \cite[Corollary 1.2.19]{CorderoRodino2020}.} By \cite[Proposition 1.2.20]{CorderoRodino2020}, $V_gf$ is also of moderate growth and, therefore, defines a tempered distribution. More generally, by interpreting
\begin{equation}
    V_gf=\cF_2\mathfrak T_{R_{st}}(f\otimes\bar g),
\end{equation}
where $\cF_2$ is the partial Fourier transform in \eqref{intro.defF2} and $\mathfrak T_{R_{st}}$ is the rescaling with
\begin{equation}
    R_{st}=\begin{pmatrix}
        O_d & I_d\\
        -I_d & I_d
    \end{pmatrix},
\end{equation}
the STFT extends to $(f,g)\in\cS'(\rd)\times\cS'(\rd)$. Importantly, this {\em metaplectic perspective} is the starting point for the definition of metaplectic Wigner distributions. The inversion formula \eqref{formulainversionSTFT} holds in $\cS'(\rd)$ whenever $f,g\in\cS'(\rd)$, see \cite[Theorem 1.2.26]{CorderoRodino2020}.

 \subsection{The formulation via the short-time Fourier transform}
\begin{definition} \label{def:modulation-space-stft}
 {Let $1\leq p,q\leq\infty$, let $m$ be a $v$-moderate weight on $\rdd$, and fix $g\in\cS(\rd)\setminus\{0\}$. $M_m^{p,q}(\rd)$ consists of $f\in\cS'(\rd)$ for which $\norm{V_gf}_{L_m^{p,q}}<\infty$.}
\end{definition}
 {When $m\equiv1$, this agrees with the unweighted notation introduced above. On the diagonal we write $M_m^p=M_m^{p,p}$.}

\begin{theorem}[Window independence]\label{thm:window-independence-stft}
Let $1\leq p,q\leq\infty$ and let $m$ be   $v$-moderate. Any two nonzero Schwartz windows define equivalent norms on $M_m^{p,q}$.  More generally, in the Banach range every nonzero $g\in M_v^1(\rd)$ is an admissible window and
\begin{equation}
 \norm{V_gf}_{L_m^{p,q}}\asymp\norm f_{M_m^{p,q}}.
\end{equation}
\end{theorem}
 \noindent See \cite[Proposition~2.3.2 and Theorem~2.3.12]{CorderoRodino2020}. This result follows by a fundamental convolution identity holding for the STFT, see \cite[Lemma 1.2.29]{CorderoRodino2020}:
 \begin{equation}
     |V_gf|\leq\frac{1}{|\langle h,\gamma\rangle|}\big(|V_hf|\ast|V_g\gamma|\big), 
 \end{equation}
 holding pointwise (it is an inequality, not an identity) for every $f\in\cS'(\rd)$ and $g,h,\gamma\in\cS(\rd)\setminus\{0\}$ such that $\langle h,\gamma\rangle\neq0$.

\subsection{Modulation spaces and other function spaces}
So far, we have discussed the inclusion relations between modulation spaces and Besov spaces. For particular choices of Lebesgue exponents and weights, we are also able to retrieve spaces appearing in harmonic analysis.

The one-sided weights $m\otimes 1$ and $1\otimes m$ explain the two basic effects measured by a modulation norm: a weight in $x$ controls decay, whereas a weight in $\xi$ controls smoothness.
In particular,
\begin{equation}\label{eq:weighted-L2-identities}
\begin{split}
 &M_{m\otimes1}^{2}(\rd)=L_m^2(\rd),\\
 &M_{1\otimes m}^{2}(\rd)=\cF^{-1}L_m^2(\rd),\\
 &M_{1\otimes v_s}^{2}(\rd)=H^s(\rd),
 \end{split}
\end{equation}
with equivalent norms, for polynomially moderate weights $m$ on $\rd$. Here, $H^s(\rd)$ denotes the potential Sobolev space. See \cite[Proposition~2.3.6]{CorderoRodino2020}. We also remark that $M^2_{v_s}(\rd)=Q_s(\rd)$ is the Shubin-Sobolev space, see the details in \cite[Lemma 4.4.19]{CorderoRodino2020}.

More generally, multiplication by $v_s(x)$ and the Bessel potential $\langle D\rangle^s$ shift, respectively, the space and frequency factors of a weight. See \cite[Theorem~2.3.14]{CorderoRodino2020}.
Finally, for the radial polynomial weight in \eqref{eq:polynomial-weight},
\begin{equation}
 \cS(\rd)=\bigcap_{s\geq0}M_{v_s}^{p,q}(\rd),\qquad
 \cS'(\rd)=\bigcup_{s\geq0}M_{1/v_s}^{p,q}(\rd),
\end{equation}
for every $1\leq p,q\leq\infty$. See \cite[Proposition~2.3.6]{CorderoRodino2020}.

\subsection{Banach structure, duality and interpolation}
  \begin{theorem}[Basic Banach properties]
 Let $1\leq p,q\leq\infty$ and let $m$ be   $v$-moderate. Then $M_m^{p,q}(\rd)$ is a Banach space, continuously embedded in $\cS'(\rd)$. If $p,q<\infty$, then $\cS(\rd)$ is dense in $M_m^{p,q}$. Moreover,
\begin{equation}
 \norm{\pi(z)f}_{M_m^{p,q}}\lesssim v(z)\norm f_{M_m^{p,q}},\qquad z\in\rdd.
\end{equation}
 \end{theorem}
 \noindent See
\cite[Proposition~2.3.8 and Theorem~2.3.9]{CorderoRodino2020}. Importantly, time-frequency shift invariance characterizes the weighted Feichtinger algebras $M^1_v$ as the smallest Banach spaces with this property.
\begin{theorem}[Minimality of $M^1_v$]
Let $v$ be a submultiplicative moderate weight. Let $(B,\norm{\cdot}_B)$ be a Banach space contained in $\cS'(\rd)$ with the following properties.
\begin{enumerate}[1.]
    \item For every $x,\xi\in\rd$
    \begin{equation}
        \norm{\pi(x,\xi)f}_B\leq v(x,\xi)\norm{f}_B, \qquad f\in B.
    \end{equation}
    \item $M^1_v\cap B\neq\{0\}$.
\end{enumerate}
Then, $M^1_v$ is embedded in $B$.    
\end{theorem}
We refer to \cite[Proposition 375]{DeGosson2011} for the proof of the previous result; see \cite{Feichtinger1981SegalAlgebra} for the original unweighted minimality of $S_0$.

\begin{theorem}[Duality] \label{thm:duality}
 Let $1\leq p,q<\infty$ and let $m$ be   $v$-moderate. Then
\begin{equation}
 (M_m^{p,q})^*\simeq M_{1/m}^{p',q'}.
\end{equation}
If $\varphi\in\cS(\rd)\setminus\{0\}$, the duality may be written as
\begin{equation}
 \langle f,h\rangle
 =\frac{1}{\|\varphi\|_2^2}\int_{\rdd}V_\varphi f(z)\overline{V_\varphi h(z)}\,dz.
\end{equation}
 \end{theorem}
 \noindent See
\cite[Theorem~2.3.10]{CorderoRodino2020}.
At endpoints it is convenient to use the closure $\cM_m^{p,q}$ of $\cS(\rd)$ in $M_m^{p,q}$. If $p,q<\infty$, then $\cM_m^{p,q}=M_m^{p,q}$.
\begin{theorem}[Complex interpolation]\label{thm:interpolation}
Let $m_j$ be moderate weights, $1\leq p_j,q_j\leq\infty$, $j=1,2$, and $0<\theta<1$. Set
\begin{equation}
 \frac{1}p=\frac{1-\theta}{p_1}+\frac{\theta}{p_2},\qquad
 \frac{1}q=\frac{1-\theta}{q_1}+\frac{\theta}{q_2},\qquad
 m=m_1^{1-\theta}m_2^\theta.
\end{equation}
Then
\begin{equation}
 (\cM_{m_1}^{p_1,q_1},\cM_{m_2}^{p_2,q_2})_{[\theta]}=\cM_m^{p,q}
\end{equation}
with equivalent norms.  In particular, when all exponents are finite the same formula holds with $M$ in place of $\cM$.
 \end{theorem}
 \noindent See
\cite[Proposition~2.3.16]{CorderoRodino2020}.

  \subsection{Embeddings, convolution and pointwise multiplication}

\begin{theorem}[General weighted embedding]\label{thm:embedding-general}
Let $1\leq p_1\leq p_2\leq\infty$, $1\leq q_1\leq q_2\leq\infty$, and assume $m_2\lesssim m_1$. Then
\begin{equation}
 M_{m_1}^{p_1,q_1}(\rd)\hookrightarrow M_{m_2}^{p_2,q_2}(\rd).
\end{equation}
\end{theorem}
\noindent This is
\cite[Theorem~2.4.17]{CorderoRodino2020}.
For the historical frequency weights there is a sharper necessary-and-sufficient criterion.
\begin{theorem}[Sharp embedding for $1\otimes v_s$]\label{thm:embedding}
 Let $1\leq p_j,q_j\leq\infty$ and $s_j\in\bR$. Then
\begin{equation}
 M_{1\otimes v_{s_1}}^{p_1,q_1}(\rd)\hookrightarrow M_{1\otimes v_{s_2}}^{p_2,q_2}(\rd)
\end{equation}
if and only if $p_1\leq p_2$ and either
\begin{equation}
 q_1\leq q_2,\quad s_1\geq s_2,
\end{equation}
or
\begin{equation}
 q_1>q_2,\quad \frac{s_1}{d}+\frac{1}{q_1}>\frac{s_2}{d}+\frac{1}{q_2}.
\end{equation}
\end{theorem}
\noindent See
\cite[Proposition~2.4.18]{CorderoRodino2020}.

The general convolution and product statements make the geometry of the weight explicit.  Let
\begin{equation}
 m_1(x)=m(x,0),\quad m_2(\xi)=m(0,\xi),\qquad
 v_1(x)=v(x,0),\quad v_2(\xi)=v(0,\xi),
\end{equation}
where $m$ is $v$-moderate, and let $\nu$ be an auxiliary polynomially moderate weight on $\rd$. All exponents in the next two statements belong to $[1,\infty]$.

\begin{theorem}[Convolution]\label{thm:convolution}
If
\begin{equation}\label{indexconv}
 \frac{1}p+1\leq\frac{1}{p_1}+\frac{1}{p_2},\qquad
 \frac{1}q\leq\frac{1}{q_1}+\frac{1}{q_2},
\end{equation}
then
\begin{equation}\label{eq:general-convolution-modspaces}
 M_{m_1\otimes\nu}^{p_1,q_1}(\rd)*
 M_{v_1\otimes(v_2\nu^{-1})}^{p_2,q_2}(\rd)
 \hookrightarrow M_m^{p,q}(\rd),
\end{equation}
and
\begin{equation}
 \norm{f*g}_{M_m^{p,q}}
 \lesssim
 \norm f_{M_{m_1\otimes\nu}^{p_1,q_1}}
 \norm g_{M_{v_1\otimes(v_2\nu^{-1})}^{p_2,q_2}}.
\end{equation}
\end{theorem}
 \noindent See
\cite[Proposition~2.4.19]{CorderoRodino2020}.
\begin{theorem}[Pointwise multiplication]\label{thm:multiplication}
If
\begin{equation}\label{condproduct}
 \frac{1}p\leq\frac{1}{p_1}+\frac{1}{p_2},\qquad
 \frac{1}q+1\leq\frac{1}{q_1}+\frac{1}{q_2},
\end{equation}
then
\begin{equation}\label{eq:general-product-modspaces}
 M_{(v_1\nu)\otimes m_2}^{p_1,q_1}(\rd)\cdot
 M_{\nu^{-1}\otimes v_2}^{p_2,q_2}(\rd)
 \hookrightarrow M_m^{p,q}(\rd),
\end{equation}
with the corresponding norm inequality.
\end{theorem}
\noindent See
\cite[Proposition~2.4.23]{CorderoRodino2020}.  
For convolution, the $x$-index follows Young's inequality and the $\xi$-index follows H\"older's inequality. For pointwise multiplication, the roles are reversed. In the unweighted case, the conditions \eqref{indexconv} and \eqref{condproduct} are also sharp, see \cite[Propositions 2.7.2 and 2.7.4]{CorderoRodino2020}. See also the results in \cite{GFWZ2016} for sharp conditions for convolution and product (and embedding) between modulation spaces with tensorial polynomial weights. 

  \subsection{Fourier transform and the importance of the weight geometry}

The Fourier transform rotates the STFT:
\begin{equation}\label{eq:fourier-stft-rotation}
 |V_{\widehat g}\widehat f(x,\xi)|=|V_gf(-\xi,x)|.
\end{equation}

This identity does \emph{not} justify a general map $M_m^{p,q}\to M_{m\circ J}^{q,p}$, where $J(x,\xi)=(\xi,-x)$, with the standard order of mixed norms.  The correct statements distinguish the diagonal case from the mixed one.

  \begin{theorem}[Fourier transform on weighted modulation spaces]
Let $1\leq p\leq\infty$ and let $u,w$ be even moderate weights on $\rd$. Then
\begin{equation}\label{eq:fourier-product-weights}
 \cF:M_{u\otimes w}^{p}(\rd)\longrightarrow M_{w\otimes u}^{p}(\rd)
\end{equation}
is a topological isomorphism.  In particular,
\begin{equation}
 \cF:M_{1\otimes v_s}^{p}\longrightarrow M_{v_s\otimes1}^{p},
\end{equation}
so a frequency-only Sobolev weight is rotated into a space-only weight.  For a radial phase-space weight $v_s$, the diagonal space $M_{v_s}^{p}$ is Fourier invariant.
\end{theorem}
 \noindent See
\cite[Theorem~2.3.27 and Proposition~2.3.28]{CorderoRodino2020}.
For $p\neq q$, the natural Fourier image is a Wiener amalgam space.  With the same product weights,
\begin{equation}\label{eq:fourier-modulation-wiener}
 \cF\big(M_{u\otimes w}^{p,q}(\rd)\big)
 =W(\cF L_u^p,L_w^q)(\rd)
\end{equation}
with equivalent norms. See
\cite[Equation~(2.54)]{CorderoRodino2020}.  This distinction is one of the main reasons to keep the full weight $m(x,\xi)$ visible in the notation.

\section{Gabor frames and discrete characterizations}
\label{sec:gabor}

\subsection{Gabor systems and frames}

We retain the polynomial control-weight assumptions of Section~\ref{sec:stft-definition}. Gabor analysis asks when the STFT can be sampled on a discrete set without losing stable reconstruction. The natural sampling sets are lattices $\Lambda\subset\rdd$, and frame theory provides the correct language because redundancy is not a defect here. It is precisely what permits stable and flexible reconstruction. Indeed, by the Balian-Low theorem, a Gabor system which is an orthonormal (or Riesz) basis cannot be generated by a window that is well localized both in time and in frequency.

For a rectangular lattice $\Lambda=\alpha\zd\times\beta\zd$, with $\alpha,\beta>0$, the samples $V_gf(\alpha m,\beta n)$ are exactly the Gabor coefficients of $f$ for the convention $\pi(x,\xi)=M_\xi T_x$. The basic question is therefore whether these coefficients determine $f$ stably and whether their sequence-space decay still measures the same modulation-space regularity as the continuous STFT norm. This continuous-to-discrete passage is one of the central manifestations of Feichtinger's local-global philosophy.

For this rectangular lattice, one seeks a dual window $\gamma$ such that
\begin{equation}
 f=\sum_{m,n\in\zd}c_{mn}\pi(\alpha m,\beta n)\gamma,
 \qquad c_{mn}=V_gf(\alpha m,\beta n).
\end{equation}
Such a reconstruction requires a frame hypothesis; it does not hold for an arbitrary window and lattice. Using $T_{\alpha m}M_{\beta n}$ instead changes only the matching phase convention for coefficients and atoms.

For a general lattice $\Lambda\subset\rdd$ and $g\in L^2(\rd)\setminus\{0\}$, the associated Gabor system is
\begin{equation}
 \cG(g,\Lambda)=\{\pi(\lambda)g \mid \lambda\in\Lambda\}.
\end{equation}
It is a Gabor frame for $L^2(\rd)$ if it satisfies the frame definition in the Hilbert space $L^2(\rd)$, i.e., if there exist constants $A,B>0$ such that
\begin{equation}
 A\norm{f}_{2}^2\leq \sum_{\lambda\in\Lambda}|\scal{f}{\pi(\lambda)g}|^2\leq B\norm{f}_{2}^2,
 \qquad f\in L^2(\rd),
\end{equation}
with Gabor coefficients
\begin{equation}
c_\lambda=\scal{f}{\pi(\lambda)g}.
\end{equation}
If $\cG(g,\Lambda)$ is a Gabor frame for $L^2(\rd)$, then $\|(\scal{f}{\pi(\lambda)g})_\lambda\|_{\ell^2}$ is an equivalent norm for $L^2(\rd)$.

Whether $\cG(g,\Lambda)$ is a frame depends on both the window and the lattice. For the Gaussian $\varphi(t)=e^{-\pi|t|^2}$ and a product lattice
\begin{equation}
 \Lambda=(\alpha_1\bZ\times\cdots\times\alpha_d\bZ)
 \times(\beta_1\bZ\times\cdots\times\beta_d\bZ),
 \qquad \alpha_j,\beta_j>0,
\end{equation}
the criterion is particularly explicit:
\begin{equation}
 \cG(\varphi,\Lambda)\text{ is a frame for }L^2(\rd)
 \quad\Longleftrightarrow\quad
 \alpha_j\beta_j<1\quad(j=1,\ldots,d).
\end{equation}
For $d=1$ this is the celebrated result of Lyubarskii and Seip-Wallst\'en \cite{Lyubarskii1992,SeipWallsten1992}; the case $d>1$ follows by tensorization. See \cite{Grochenig2001,CorderoRodino2020} for the frame theory used here.

Given a lattice $\Lambda\subset\rdd$ and $g\in\cS(\rd)$, we can consider the analysis operator
\begin{equation}
C_gf=(\langle f,\pi(\lambda)g\rangle)_{\lambda\in\Lambda}, \quad \text{ for } f\in L^2(\rd),
\end{equation}
the synthesis operator
\begin{equation}
D_gc=\sum_{\lambda\in\Lambda}c_\lambda\pi(\lambda)g, \quad \text{ for } c\in \ell^2(\Lambda),
\end{equation}
and the Gabor frame operator
\begin{equation}
S_{g,g}f=D_gC_gf=\sum_{\lambda\in\Lambda}\langle f,\pi(\lambda)g\rangle \pi(\lambda)g.
\end{equation}

\begin{theorem}[Reproducing formulas for Gabor frames]
Assume that $\cG(g,\Lambda)$ is a Gabor frame for $L^2(\rd)$ with frame bounds $A$ and $B$. Let $\gamma:=S^{-1}g$ be the {\em canonical dual window} of $g$.
\begin{enumerate}[(i)]
    \item The Gabor frame operator $S$ is a topological isomorphism of $L^2(\rd)$ onto itself, self-adjoint and positive, with
    \begin{equation}
        A\,{\mathrm{id}_{L^2}}\leq S\leq B\,{\mathrm{id}_{L^2}},
    \end{equation}
    equivalently, $A\|f\|_2^2\leq\langle Sf,f\rangle\leq B\|f\|_2^2$.
    \item $S^{-1}$ is a topological isomorphism, self-adjoint and positive, and
    \begin{equation}
        B^{-1}\,{\mathrm{id}_{L^2}}\leq S^{-1}\leq A^{-1}\,{\mathrm{id}_{L^2}}.
    \end{equation}
    The corresponding quadratic-form bounds are the frame inequalities for the {\em canonical dual frame}. Therefore, $\cG(\gamma,\Lambda)=\{\pi(\lambda)\gamma \mid \lambda\in\Lambda\}$ is a Gabor frame for $L^2(\rd)$, with frame bounds $0<B^{-1}\leq A^{-1}$ (recall that $S$ commutes with $\pi(\lambda)$, $\lambda\in\Lambda$, so that $S^{-1}\pi(\lambda)g=\pi(\lambda)\gamma$).
    \item For $f\in L^2(\rd)$, we have the reproducing formulas
    \begin{equation}
    \label{eq:reconstruction-formulas}
        f=\sum_{\lambda\in\Lambda}\scal{f}{\pi(\lambda)\gamma}\pi(\lambda)g, \quad f=\sum_{\lambda\in\Lambda}\scal{f}{\pi(\lambda)g}\pi(\lambda)\gamma,
    \end{equation}
    and these series converge unconditionally in the norm of $L^2(\rd)$.
    \item If $A=B$, then $S=A\,\mathrm{id}_{L^2}$, $S^{-1}=A^{-1}\,\mathrm{id}_{L^2}$ and, for all $f\in L^2(\rd)$,
    \begin{equation}
        f=\frac{1}{A}\sum_{\lambda\in\Lambda}\scal{f}{\pi(\lambda)g}\pi(\lambda)g.
    \end{equation}
\end{enumerate}
\end{theorem}

The relations \eqref{eq:reconstruction-formulas} give the reconstruction formulas via the canonical dual frame $\cG(\gamma,\Lambda)=\{\pi(\lambda)\gamma \mid \lambda\in\Lambda\}$, which is still a Gabor frame.

\subsection{Discrete modulation-space norms}

We next compare the continuous STFT norm with a norm of its samples. The order of the discrete variables must agree with the order in $L_m^{p,q}$.

For $\Lambda=\alpha\zd\times\beta\zd$, set
\begin{equation}\label{eq:discrete-mixed-norm}
 \|c\|_{\ell_m^{p,q}(\Lambda)}
 =\left(\sum_{n\in\zd}\left(\sum_{k\in\zd}
 |c_{\alpha k,\beta n}|^p m(\alpha k,\beta n)^p\right)^{q/p}\right)^{1/q},
\end{equation}
with the usual changes when an exponent is infinite. For a general lattice we use the geometric sequence norm
\begin{equation}\label{eq:geometric-sequence-norm}
 \|c\|_{\ell_m^{p,q}(\Lambda)}
 :=\left\|\sum_{\lambda\in\Lambda}|c_\lambda|\,\boldsymbol{1}_{\lambda+Q}\right\|_{L_m^{p,q}},
\end{equation}
where $Q$ is a bounded neighborhood of the origin such that $\Lambda+Q=\rdd$. Different choices of $Q$ give equivalent norms; for a rectangular lattice this agrees with \eqref{eq:discrete-mixed-norm} up to equivalence. This convention fixes the physical time-frequency order rather than an arbitrary enumeration of the lattice; see \cite{Grochenig2001} for the associated sequence-space construction.

At endpoints, weak-$\ast$ convergence of expansions will mean convergence in $M_{1/v}^\infty=(M_v^1)^*$, tested against $M_v^1$. This convention also applies to the quasi-Banach extensions below.

The reconstruction formulas \eqref{eq:reconstruction-formulas} extend to modulation spaces through the mixed Gabor frame operator
\begin{equation}
S_{g,\gamma}f=D_\gamma C_gf=\sum_{\lambda\in\Lambda}\langle f,\pi(\lambda)g\rangle \pi(\lambda)\gamma.
\end{equation}

The analysis map sends a continuous modulation norm to its discrete counterpart. See \cite[Theorem~3.2.32]{CorderoRodino2020}.
\begin{theorem}
Let $g\in M^1_v(\rd)$ and let $m$ be a $v$-moderate weight. Then the analysis operator $C_g$ is bounded from $M_m^{p,q}(\rd)$ to $\ell_m^{p,q}(\Lambda)$, for every $1\leq p,q\leq\infty$ and for every lattice $\Lambda\subset\rdd$, with operator norm
\begin{equation}
\|C_g\|\leq C(p,q,m,v,\Lambda)\|g\|_{M^1_v},
\end{equation}
where the constant may depend on the exponents, the weight bounds, and the lattice.
\end{theorem}

Analogously, we have the continuity of the synthesis operator, see
\cite[Theorem~3.2.33]{CorderoRodino2020}.
\begin{theorem}
Let $g\in M^1_v(\rd)$ and let $m$ be a $v$-moderate weight. Then the synthesis operator $D_g$ is bounded from $\ell_m^{p,q}(\Lambda)$ to $M_m^{p,q}(\rd)$, for every $1\leq p,q\leq\infty$ and for every lattice $\Lambda\subset\rdd$, with operator norm
\begin{equation}
\|D_g\|\leq C(p,q,m,v,\Lambda)\|g\|_{M^1_v},
\end{equation}
where the constant may depend on the exponents, the weight bounds, and the lattice. Moreover,
\begin{equation}
D_gc=\sum_{\lambda\in\Lambda}c_\lambda\pi(\lambda)g,
\end{equation}
with unconditional convergence in $M_m^{p,q}$ if $p,q<\infty$; at an infinite exponent, convergence is weak-$\ast$ in $M_{1/v}^\infty$.
\end{theorem}

Consequently, we have the continuity of the Gabor frame operator on modulation spaces.
\begin{theorem}
Let $g,\gamma\in M^1_v(\rd)$ and let $m$ be a $v$-moderate weight. Then the Gabor frame operator $S_{g,\gamma}=D_\gamma C_g$ is bounded on $M_m^{p,q}(\rd)$, for every $1\leq p,q\leq\infty$ and for every lattice $\Lambda\subset\rdd$, with operator norm
\begin{equation}
\|S_{g,\gamma}\|\leq C(p,q,m,v,\Lambda)\|g\|_{M^1_v}\|\gamma\|_{M^1_v},
\end{equation}
where the constant may depend on the exponents, the weight bounds, and the lattice.
\end{theorem}

In the particular case where $S_{g,\gamma}=D_\gamma C_g=\mathrm{id}_{L^2}$, we have the reconstruction formulas \eqref{eq:reconstruction-formulas} on modulation spaces.

\begin{theorem}
\label{thm:atom-decomposition-Banach}
Let $1\leq p,q\leq\infty$, $m$ be a $v$-moderate weight, and $g,\gamma\in\cS(\rd)$. If $S_{g,\gamma}=D_\gamma C_g=\mathrm{id}_{L^2}$, then
\begin{equation}
\label{eq:gabor-expansion-Banach}
f=\sum_{\lambda\in\Lambda}\scal{f}{\pi(\lambda)\gamma}\pi(\lambda)g, \quad f=\sum_{\lambda\in\Lambda}\scal{f}{\pi(\lambda)g}\pi(\lambda)\gamma,
\end{equation}
with unconditional convergence in $M_m^{p,q}$ if $p,q<\infty$. At an infinite exponent, convergence is weak-$\ast$ in $M_{1/v}^\infty$. Furthermore,
\begin{equation}
 \|f\|_{M_m^{p,q}}
 \asymp\|C_gf\|_{\ell_m^{p,q}(\Lambda)}
 \asymp\|C_\gamma f\|_{\ell_m^{p,q}(\Lambda)}.
\end{equation}
In particular, $\|f\|_{M_m^{p,q}}\leq\|D_\gamma\|\|C_gf\|_{\ell_m^{p,q}}$ and $\|C_gf\|_{\ell_m^{p,q}}\leq\|C_g\|\|f\|_{M_m^{p,q}}$. The analogous bounds for $C_\gamma$ use $D_g$ and $C_\gamma$, respectively.
\end{theorem}

As a corollary, we have a characterization of Schwartz functions via their Gabor coefficients.
\begin{corollary}
Let $g\in\cS(\rd)$ and suppose that $\cG(g,\Lambda)$ is a Gabor frame for $L^2(\rd)$. Then
\begin{equation}
f\in\cS(\rd) \quad\Longleftrightarrow\quad \sup_{\lambda\in\Lambda}\langle\lambda\rangle^N|\langle f,\pi(\lambda)g\rangle|<\infty, \ \text{ for all } N\in\bN.
\end{equation}
\end{corollary}

The decay and summability of the Gabor coefficients characterize modulation spaces, see
\cite[Theorems~3.2.35-3.2.37]{CorderoRodino2020}.

\begin{theorem}[Gabor frame characterization]
\label{thm:gabor-characterization-Banach}
Let $1\leq p,q\leq\infty$, $m$ be a $v$-moderate weight, and $g\in\cS(\rd)\setminus\{0\}$. If $\cG(g,\Lambda)$ is a Gabor frame, then
\begin{equation}
 f\in M_m^{p,q}(\rd)
 \quad\quad\Longleftrightarrow\quad\quad
 \big(\scal{f}{\pi(\lambda)g}\big)_{\lambda\in\Lambda}\in \ell_m^{p,q}(\Lambda),
\end{equation}
with the following norm equivalence:
\begin{equation}
 \norm{f}_{M_m^{p,q}}\asympconst
 \norm{(\scal{f}{\pi(\lambda)g})_{\lambda\in\Lambda}}_{\ell_m^{p,q}}.
\end{equation}
\end{theorem}

\section{Gabor analysis of operators}
\label{sec:gabor-operators}
Let $T:\cS(\rd)\to\cS'(\rd)$ be a continuous linear operator and fix $g\in\cS(\rd)\setminus\{0\}$. Its Gabor matrix is
\begin{equation}
\label{eq:gabor-matrix}
 K_T(w,z)=\scal{T\pi(z)g}{\pi(w)g},\qquad z,w\in\rdd.
\end{equation}
This matrix is the kernel of the operator obtained by expressing $T$ in STFT coordinates. Indeed, if $\|g\|_2=1$, then the inversion formula for the short-time Fourier transform reads as $V_g^*V_g=\mathrm{id}_{L^2}$, and thus we can write the operator $T$ as
\begin{equation}
T=V_g^*V_gTV_g^*V_g.
\end{equation}
The linear operator $V_gTV_g^*$ is an integral operator having $K_T$ as kernel.

This point of view is very useful as it helps studying the boundedness of $T$ on modulation spaces. In particular, by definition and the inversion formula, $V_g$ is a bounded linear operator from $M_m^{p,q}$ to $L_m^{p,q}$, and $V_g^*$ from $L_m^{p,q}$ to $M_m^{p,q}$. Hence, boundedness of $T$ on modulation spaces can be inferred by the corresponding boundedness of $V_gTV_g^*$ on mixed-norm $L_m^{p,q}$ spaces.

For lattice points $\lambda,\mu\in\Lambda$, one obtains the discrete matrix
\begin{equation}
 K_T(\mu,\lambda)=\scal{T\pi(\lambda)g}{\pi(\mu)g}.
\end{equation}

\subsection{The Sj\"ostrand classes}
Important instances of Gabor matrices associated to linear and continuous operators are those deriving from the quantization of a function $\sigma:\rdd\to\bC$ belonging to a suitable class of symbols.

In 1994 Sj\"ostrand \cite{Sjoestrand1994} introduced a symbol class described by time-frequency concentration on phase space, subsequently identified with the modulation space $M^{\infty,1}(\rdd)$, already introduced by Feichtinger in 1983 \cite{Feichtinger1983}. Explicitly, a symbol $\sigma\in\cS'(\rdd)$ belongs to the Sj\"ostrand class $M^{\infty,1}(\rdd)$ if
\begin{equation}
\int_{\rdd}\sup_{z\in\rdd}|V_\Phi\sigma(z,\zeta)|d\zeta<\infty,
\end{equation}
with respect to a window $\Phi\in\cS(\rdd)$.

The Weyl quantization of a symbol $\sigma$ is the linear operator $\Opw(\sigma)$ given by:
\begin{equation} \label{eq:Weyl-quantization}
\Opw(\sigma)f(x)=\int_{\rdd}\sigma\left(\frac{x+y}{2},\xi\right)e^{2\pi i(x-y)\cdot\xi}f(y)dyd\xi.
\end{equation}
Sj\"ostrand proved the following fundamental results about the Weyl transform of symbols in $M^{\infty,1}(\rdd)$ \cite{Sjoestrand1994,Sjoestrand1995}.
\begin{enumerate}[(i)]
    \item If $\sigma\in M^{\infty,1}(\rdd)$, then $\Opw(\sigma)$ is bounded on $L^2(\rd)$.
    \item If $\sigma_1,\sigma_2\in M^{\infty,1}(\rdd)$ and $\tau$ is a symbol such that $\Opw(\tau)=\Opw(\sigma_1)\Opw(\sigma_2)$, then $\tau\in M^{\infty,1}(\rdd)$. Thus the corresponding Weyl operators form a Banach algebra under composition; the symbol space is an algebra for the induced Weyl product, also called the Sj\"ostrand algebra.
    \item If $\sigma\in M^{\infty,1}(\rdd)$ and $\Opw(\sigma)$ is invertible on $L^2(\rd)$, then $\Opw(\sigma)^{-1}=\Opw(\tau)$ for some $\tau\in M^{\infty,1}(\rdd)$. This is the Wiener property of $M^{\infty,1}(\rdd)$.
\end{enumerate}

Arbitrary modulation spaces as symbol classes for pseudodifferential operators were introduced in \cite{GroechenigHeil1999} independently of Sj\"ostrand's work.
\subsection{Quasi-diagonalization of Gabor matrices}
Gr\"ochenig's time-frequency approach \cite{Grochenig2006} identifies a common mechanism behind these properties: the Gabor matrix of a Weyl operator is controlled by an integrable envelope away from the diagonal. This converts boundedness into a sequence-space estimate and connects composition and inversion with Banach algebras of matrices. It also makes the passage to weights transparent, because the envelope can be measured with the same control weight as the modulation space.

The gain is therefore not merely a different proof of $L^2$ boundedness. The matrix description shows simultaneously how an operator acts on the full modulation scale and how its off-diagonal concentration is encoded in the regularity of its symbol.

If $\sigma:\rdd\to\bC$ is a symbol, then the Gabor matrix of its Weyl quantization $\Opw(\sigma)$ with respect to a window $g\in\cS(\rd)$ is
\begin{equation}
K(w,z)=\scal{\Opw(\sigma)\pi(z)g}{\pi(w)g},\qquad w,z\in\rdd.
\end{equation}

The next result on almost diagonalization is crucial and all properties of the Sj\"ostrand class follow from it. It allows us to characterize symbols $\sigma$ in the weighted Sj\"ostrand class by suitable decay of the Gabor matrix of its Weyl quantization. Let us introduce the
standard symplectic rotation on phase space,

\begin{equation}
 j(x,\xi)=(\xi,-x),\qquad (x,\xi)\in\rdd.
\end{equation}
The symbol $\sigma$ is a distribution on $\rdd$, so its STFT is defined on $\bR^{4d}$. Accordingly, the notation $M_{1\otimes(v\circ j^{-1})}^{\infty,q}(\rdd)$ means that the weight is imposed on the second $2d$-dimensional STFT variable, not on the symbol variable itself.

\begin{theorem}[Almost diagonalization]
\label{thm:almost-diagonalization}
Let $g\in\cS(\rd)\setminus\{0\}$ and $\Lambda\subset\rdd$ be a lattice. Assume that $\cG(g,\Lambda)$ is a Gabor frame. Then the following properties are equivalent.
\begin{enumerate}[(i)]
    \item $\sigma\in M_{1\otimes(v\circ j^{-1})}^{\infty,1}(\rdd)$.
    \item $\sigma\in\cS'(\rdd)$ and there exists a function $H\in L^1_v(\rdd)$ such that
    \begin{equation}
    |\scal{\Opw(\sigma)\pi(z)g}{\pi(w)g}|\leq H(w-z),\qquad z,w\in\rdd.
    \end{equation}
    \item $\sigma\in\cS'(\rdd)$ and there exists a sequence $h\in\ell^1_v(\Lambda)$ such that
    \begin{equation}
    |\scal{\Opw(\sigma)\pi(\mu)g}{\pi(\lambda)g}|\leq h(\lambda-\mu),\qquad \lambda,\mu\in\Lambda.
    \end{equation}
\end{enumerate}
\end{theorem}

In particular, for the unweighted case, symbols $\sigma$ in the Sj\"ostrand class $M^{\infty,1}$ are characterized by the almost diagonalization of the Gabor matrix of $\Opw(\sigma)$ for a suitable function $H\in L^1(\rdd)$ or a suitable sequence $h\in\ell^1(\Lambda)$.

An immediate consequence of Theorem \ref{thm:almost-diagonalization} is that any linear and continuous operator $T:\cS(\rd)\to\cS'(\rd)$ satisfying the almost diagonalization estimate is the Weyl quantization of a symbol in the Sj\"ostrand class. This characterizes the Weyl operators with symbols in the Sj\"ostrand class.

\begin{corollary}
Let $T:\cS(\rd)\to\cS'(\rd)$ be a continuous linear operator, $g\in\cS(\rd)\setminus\{0\}$, and let $\Lambda\subset\rdd$ be a lattice. Assume that $\cG(g,\Lambda)$ is a Gabor frame, and that $T$ satisfies the estimate
\begin{equation}
|\scal{T\pi(\mu)g}{\pi(\lambda)g}|\leq h(\lambda-\mu),\qquad \lambda,\mu\in\Lambda,
\end{equation}
for some sequence $h\in\ell_v^1(\Lambda)$. Then there exists a symbol $\sigma\in M_{1\otimes(v\circ j^{-1})}^{\infty,1}(\rdd)$ such that $T=\Opw(\sigma)$.
\end{corollary}

A generalization of Theorem \ref{thm:almost-diagonalization} to different modulation spaces $M_{1\otimes(v\circ j^{-1})}^{\infty,q}(\rdd)$ is discussed in \cite{BastianoniCordero2022}. Moreover, the envelope estimate in Theorem~\ref{thm:almost-diagonalization} yields the following boundedness result.

\begin{theorem}
Let $\sigma\in M_{1\otimes(v\circ j^{-1})}^{\infty,1}(\rdd)$ and let $m$ be a $v$-moderate weight. Then $\Opw(\sigma)$ is bounded on $M_m^{p,q}(\rd)$ for all $1\leq p,q\leq\infty$, and its operator norm can be estimated uniformly by
\begin{equation}
\|\Opw(\sigma)\|\leq C\|\sigma\|_{M_{1\otimes(v\circ j^{-1})}^{\infty,1}},
\end{equation}
where $C$ may depend on the exponents and the moderation bounds of the weights.
\end{theorem}

In particular, for $\sigma\in M^{\infty,1}(\rdd)$, $\Opw(\sigma)$ is bounded on $L^2(\rd)$ and $M^{p,q}(\rd)$ for all $1\leq p,q\leq\infty$.

\section{Quasi-Banach modulation spaces}
\label{sec:quasi-banach}

The quasi-Banach theory is best viewed as a continuation of the Gabor discretization developed above rather than as a separate branch of the subject. The STFT and the coefficient spaces remain the same. What changes is the functional-analytic setting, especially the loss of local convexity and the need for more restrictive universal window classes. We therefore keep the atomic and coefficient characterizations here, immediately after the Banach and operator-theoretic Gabor theory.

The quasi-Banach regime $0<p<1$ or $0<q<1$ is essential for sparsity, nonlinear approximation, and several refined boundedness problems, and the modulation space norm is useful as a flexible lower bound in uncertainty principles. Although  {the STFT quasi-norm in Definition~\ref{def:modulation-space-stft}} still makes sense, several tools from Banach-space theory are no longer available.

\subsection{Quasi-norms and $p$-triangle inequalities}

The quasi-Banach theory of modulation spaces $M_m^{p,q}$ was developed by Galperin and Samarah in \cite{GalperinSamarah2004}, who showed that modulation spaces admit an atomic decomposition and are characterized by summability and decay properties of Gabor coefficients.

Recall that a quasi-Banach space $X$ is a vector space equipped with a quasi-norm $\|\cdot\|:X\to[0,\infty)$ which satisfies the following properties for all $x,y\in X$ and scalars $a$:
\begin{enumerate}
    \item $\|x\|=0 \Leftrightarrow x=0$.
    \item $\|ax\|=|a|\|x\|$.
    \item quasi-triangle inequality: there exists a constant $C\geq1$ (independent of $x$ and $y$) such that $\|x+y\|\leq C(\|x\|+\|y\|)$.
    \item $X$ is complete with respect to $\|\cdot\|$.
\end{enumerate}

Important instances of quasi-Banach spaces are the $L^p$-spaces for $0<p<1$, which satisfy the quasi-triangle inequality with constant $C=2^{1/p-1}$, and the $\ell^p$-spaces for $0<p<1$.

For mixed norms a useful common exponent is $r=\min\{1,p,q\}$. The $r$-triangle inequality reads
\begin{equation}\label{eq:r-triangle}
 \|F+H\|_{L_m^{p,q}}^r\leq\|F\|_{L_m^{p,q}}^r+\|H\|_{L_m^{p,q}}^r.
\end{equation}
It replaces the ordinary triangle inequality when one of the exponents is below one. The same estimate passes to modulation quasi-norms because the STFT is linear in the signal.

We keep Definition~\ref{def:modulation-space-stft} unchanged for $0<p,q\leq\infty$. When $p<1$ or $q<1$, the resulting space is quasi-Banach rather than Banach. Thus the object being measured is the same STFT, while the functional-analytic tools available for manipulating its mixed quasi-norm change.

Some basic properties of modulation spaces with $1\leq p,q\leq\infty$ still hold in the quasi-Banach setting, such as the independence of the window used to compute the quasi-norm.

\begin{theorem}
Let $0<p,q\leq\infty$ and let $m$ be a $v$-moderate weight. If $g_1,g_2\in\cS(\rd)\setminus\{0\}$, then
\begin{equation}
 \norm{V_{g_1}f}_{L_m^{p,q}}\asympconst \norm{V_{g_2}f}_{L_m^{p,q}},\qquad f\in\cS'(\rd).
\end{equation}
\end{theorem}

In the Banach range, every nonzero window in $M_v^1$ is admissible. Below one, admissible classes depend on the exponents and stronger control of the window is needed: for instance, windows in $M^r_v(\rd)$ with $r=\min\{1,p,q\}$ are admissible \cite{GalperinSamarah2004}. Schwartz windows provide a common class for all $0<p,q\leq\infty$ under the polynomial weight assumptions, but this does not make the corresponding operator bounds uniform in $p$ and $q$.

\begin{theorem}
Let $0<p_1\leq p_2\leq\infty$, $0<q_1\leq q_2\leq\infty$, and let $m$ be a $v$-moderate weight. Then we have the continuous embedding
\begin{equation}
M_m^{p_1,q_1}(\rd)\hookrightarrow M_m^{p_2,q_2}(\rd).
\end{equation}
\end{theorem}

\subsection{Atomic and frame methods in the quasi-Banach setting}

The Hilbert-space reconstruction identities also underpin the quasi-Banach coefficient characterizations. The issue is to extend the analysis and synthesis maps to the corresponding mixed quasi-norms.

The analysis and synthesis maps remain bounded with Schwartz windows; their constants must now be allowed to depend on the exponents and on the window. See \cite{GalperinSamarah2004} and \cite[Theorem~8.3]{Rauhut2005}.

\begin{theorem}[Gabor operators in the quasi-Banach regime]\label{thm:gabor-operators-quasi}
Let $0<p,q\leq\infty$, let $m$ be a $v$-moderate weight with polynomially bounded control $v$, and let $\Lambda\subset\rdd$ be a lattice. If $g\in\cS(\rd)$, then
\begin{equation}
 C_g:M_m^{p,q}\longrightarrow\ell_m^{p,q}(\Lambda),
 \qquad D_g:\ell_m^{p,q}(\Lambda)\longrightarrow M_m^{p,q}
\end{equation}
are bounded. In particular, there are finite constants $C_{g,p,q,m,v,\Lambda}$ and $D_{g,p,q,m,v,\Lambda}$ such that
\begin{equation}
 \|C_gf\|_{\ell_m^{p,q}}\leq C_{g,p,q,m,v,\Lambda}\|f\|_{M_m^{p,q}},
 \qquad
 \|D_gc\|_{M_m^{p,q}}\leq D_{g,p,q,m,v,\Lambda}\|c\|_{\ell_m^{p,q}}.
\end{equation}
The synthesis series converges unconditionally in $M_m^{p,q}$ for $p,q<\infty$, and weak-$\ast$ in $M_{1/v}^\infty$ at an infinite exponent. If also $\gamma\in\cS(\rd)$, then $S_{g,\gamma}=D_\gamma C_g$ is bounded on $M_m^{p,q}$.
\end{theorem}

The distinction is substantive: a bound for $D_g$ controlled uniformly by $\|g\|_{M_v^1}$ would, on applying $D_g$ to a single nonzero coefficient, force an unavailable uniform estimate of the $M_m^{p,q}$ norm of the window by its $M_v^1$ norm.

In the particular case where $S_{g,\gamma}=D_\gamma C_g=\mathrm{id}_{L^2}$, we obtain the reconstruction formulas \eqref{eq:reconstruction-formulas} on modulation spaces for the quasi-Banach setting, see \cite[Theorem~3.7]{GalperinSamarah2004}.

\begin{theorem}
\label{thm:atom-decomposition-quasiBanach}
Let $0<p,q\leq\infty$, $m$ be a $v$-moderate weight, and $g,\gamma\in\cS(\rd)$. If the Gabor frame operator $S_{g,\gamma}=D_\gamma C_g=\mathrm{id}_{L^2}$, then
\begin{equation}
\label{eq:gabor-expansion-quasiBanach}
f=\sum_{\lambda\in\Lambda}\scal{f}{\pi(\lambda)\gamma}\pi(\lambda)g, \quad f=\sum_{\lambda\in\Lambda}\scal{f}{\pi(\lambda)g}\pi(\lambda)\gamma,
\end{equation}
with unconditional convergence in $M_m^{p,q}$ if $p,q<\infty$. At an infinite exponent, convergence is weak-$\ast$ in $M_{1/v}^\infty$. Furthermore,
\begin{equation}
 \|f\|_{M_m^{p,q}}
 \asymp\|C_gf\|_{\ell_m^{p,q}(\Lambda)}
 \asymp\|C_\gamma f\|_{\ell_m^{p,q}(\Lambda)},
\end{equation}
with constants depending on the windows, exponents, weights, and lattice. Each lower bound follows by applying the appropriate synthesis operator to the reconstruction identity.
\end{theorem}

This shows that modulation spaces can be characterized by summability and decay properties of Gabor coefficients $c_\lambda$ in Gabor expansions of signals $f$, providing a natural setting for time-frequency analysis. Once again, we notice that the windows $g$ and $\gamma$ are Schwartz functions.

Since the finite sequences are dense in $\ell_m^{p,q}(\Lambda)$ when $0<p,q<\infty$, Theorem \ref{thm:atom-decomposition-quasiBanach} implies that, in this case, $\cS(\rd)$ is a dense subspace of $M_m^{p,q}(\rd)$. For a rectangular lattice $\Lambda=\alpha\zd\times\beta\zd$, the related lower synthesis bound is naturally formulated on the adjoint lattice $\Lambda^\circ=\beta^{-1}\zd\times\alpha^{-1}\zd$, see \cite[Theorem~3.8]{GalperinSamarah2004}.

\begin{theorem}[Lower synthesis bound on the adjoint lattice]
Let $0<p,q<\infty$, $\Lambda=\alpha\zd\times\beta\zd$, let $m$ be a polynomially moderate weight, and let $g\in\cS(\rd)$. If $\cG(g,\Lambda)$ is a Gabor frame, then
\begin{equation}
 \left\|\sum_{k,n\in\zd}c_{k,n}\pi(k/\beta,n/\alpha)g\right\|_{M_m^{p,q}}
 \geq A\|c\|_{\ell_m^{p,q}(\Lambda^\circ)},
\end{equation}
where $A>0$ is independent of $c\in\ell_m^{p,q}(\Lambda^\circ)$. The coefficient weight is evaluated at the actual atom positions $(k/\beta,n/\alpha)$.
\end{theorem}

Stable reconstruction and mixed summability of the coefficients give the following characterization, see \cite[Theorem~8.3]{Rauhut2005}.

\begin{theorem}
\label{thm:gabor-characterization-quasi-Banach}
Let $0<p,q<\infty$, $m$ be a $v$-moderate weight, and $g\in\cS(\rd)\setminus\{0\}$. If $\cG(g,\Lambda)$ is a Gabor frame, then
\begin{equation}
 f\in M_m^{p,q}(\rd)
 \quad\quad\Longleftrightarrow\quad\quad
 \big(\scal{f}{\pi(\lambda)g}\big)_{\lambda\in\Lambda}\in \ell_m^{p,q}(\Lambda),
\end{equation}
with the following norm equivalence:
\begin{equation}
 \norm{f}_{M_m^{p,q}}\asympconst
 \norm{(\scal{f}{\pi(\lambda)g})_{\lambda\in\Lambda}}_{\ell_m^{p,q}}.
\end{equation}
\end{theorem}

Notice that the window $g$ belongs to the Schwartz class, as we are dealing with the quasi-Banach extension.

\section{Weights, Gelfand-Shilov spaces and ultradistributions}
\label{sec:weights-ultra}
The theory of modulation spaces we have developed so far only deals with weights of polynomial growth. We continue to use the notation $v_s$ from \eqref{eq:polynomial-weight}, with the ambient dimension determined by context.

In this context, the pair $(\cS,\cS')$ is a very convenient framework, and it is sufficient for several kinds of applications, but it is too restrictive for the treatment of general weight functions. Yet the extension to non-polynomial weights is highly desirable in time-frequency analysis. For instance, in the theory of window design one often attempts to construct windows that decay exponentially in time and frequency, and such windows are in the modulation space $M^1_v$ with respect to an exponential weight $v$.

The change of weights affects both test functions and generalized functions. To see the first effect, fix a Gaussian window and consider
\begin{equation}
 v(z)=e^{a|z|^\beta},\qquad a>0,\quad 0<\beta\leq1.
\end{equation}
For every $s\geq0$, polynomial growth is dominated by this weight, with a constant depending on $s,a,\beta$. Consequently,
\begin{equation}\label{eq:inequality-modulation-superexponential}
 \|f\|_{M_{v_s}^\infty}
 =\sup_{z\in\rdd}|V_gf(z)|\langle z\rangle^s
 \lesssim_{s,a,\beta}\sup_{z\in\rdd}|V_gf(z)|e^{a|z|^\beta}
 =\|f\|_{M_v^\infty}.
\end{equation}
Thus a finite rapidly weighted STFT norm forces
\begin{equation}
 M_v^\infty\subseteq\bigcap_{s\geq0}M_{v_s}^\infty=\cS(\rd).
\end{equation}
Arbitrary Schwartz functions need not belong to the space being defined, so they can no longer serve as a universal dense test class.

The reciprocal weight illustrates the other effect. Every tempered distribution has a polynomially bounded STFT and hence finite $M_{1/v}^\infty$ norm. If the ambient class is kept equal to $\cS'$, this condition merely recovers all of $\cS'$; it cannot reveal generalized functions outside it. To obtain a complete theory accommodating both rapid decay and rapid growth, one replaces $(\cS,\cS')$ by a smaller test space and a larger dual space.

Compactly supported smooth test functions do not solve this problem, since the space of compactly supported smooth functions $\cD$ is not Fourier invariant and Fourier transformation is not defined on all of $\cD'$. Gelfand-Shilov spaces provide Fourier-compatible alternatives. This leads to the ultra-modulation framework of \cite{PilipovicTeofanov2002} and to the test-function constructions discussed after the coorbit approach in \cite{FeichtingerGrochenig1989}.

It is important to distinguish the growth regimes. The weight above is submultiplicative for $0<\beta\leq1$: it is subexponential for $\beta<1$ and exponential for $\beta=1$. For $\beta>1$, the superexponential weight $e^{a|z|^\beta}$ is not moderate with respect to a locally bounded submultiplicative control weight. Indeed, the ratio $v(z+w)/v(z)$ is unbounded for a suitable fixed $w$. Such weights can occur in Gelfand-Shilov characterizations, but they are not automatically covered by the moderate-weight definition below.

Under the standing local boundedness assumptions, submultiplicative weights on Euclidean space grow at most exponentially. Accordingly, allowing arbitrary moderate weights here means removing the polynomial-growth restriction, not allowing every superexponential weight. A fixed Fourier-invariant test class such as $\Sigma^1_1$ is small enough to handle all these moderate weights; its dual supplies the required ambient ultradistributions.

\subsection{Gelfand-Shilov spaces through time-frequency decay}

A Schwartz function and all its derivatives decay faster than every polynomial. Gelfand-Shilov spaces refine this information by quantifying the growth of the constants in the derivative and moment estimates. Their parameters control decay in both the spatial and Fourier variables.

\begin{definition} \label{gelfand-shilov}
    Let $r,s\geq0$ and $A,B>0$. The Gelfand-Shilov-type space $\cS^{s,A}_{r,B}(\rd)$ is defined as follows:
    \begin{equation}
        \cS^{s,A}_{r,B}(\rd):=\left\{f\in\cS(\rd) \ \middle | \ \sup_{x\in\rd,\,\alpha,\beta\in\bN^d}\frac{|x^\alpha\partial^\beta f(x)|}{A^{|\alpha|}B^{|\beta|}(\alpha!)^r(\beta!)^s}<\infty\right\}.
    \end{equation}
    Endow this space with the supremum norm appearing in the definition. The Beurling and Roumieu spaces are obtained, respectively, by taking the projective and inductive limits
    \begin{equation}
        \Sigma^s_r(\rd):=\projlim_{A>0,B>0}\cS^{s,A}_{r,B}(\rd), \qquad \cS^s_r(\rd):=\indlim_{A>0,B>0}\cS^{s,A}_{r,B}(\rd).
    \end{equation}
\end{definition}

It can be shown that the space $\cS^s_r(\rd)$ is nontrivial if and only if $r+s>1$, or $r+s=1$ and $r,s>0$. Hence, the smallest nontrivial space with $r=s$ is $\cS^{1/2}_{1/2}(\rd)$. Some examples of functions in $\cS^{1/2}_{1/2}(\rd)$ are $P(x)e^{-a|x|^2}$, where $P$ is a polynomial on $\rd$ and $a>0$. Let us notice the trivial inclusions
\begin{equation}
    \cS^{s_1}_{r_1}(\rd)\subseteq\cS^{s_2}_{r_2}(\rd), \quad \text{ for } s_1\leq s_2 \text{ and } r_1\leq r_2.
\end{equation}
Moreover, if $f\in\cS^s_r(\rd)$, then the same holds for $x^\delta\partial^\gamma f$, for any $\gamma,\delta$.

The next theorem shows the action of the Fourier transform on $\cS^s_r(\rd)$: it interchanges the indices $s$ and $r$.

\begin{theorem}
Let $f\in\cS(\rd)$. Then $f\in\cS^s_r(\rd)$ if and only if $\widehat{f}\in\cS^r_s(\rd)$.
\end{theorem}

Consequently, for $r=s$, the spaces $\cS^s_s(\rd)$ are invariant under the action of the Fourier transform.

The next theorem shows a characterization of functions in the Gelfand-Shilov spaces in terms of the $L^\infty$ norm of $f$ and of its Fourier transform.

\begin{theorem}
\label{thm:Gelfand-Shilov-characterization}
Suppose $s,r>0$ and $s+r\geq1$. For $f\in\cS(\rd)$, the following conditions are equivalent.
\begin{enumerate}[(i)]
    \item $f\in\cS^s_r(\rd)$.
    \item There exist constants $A,B>0$ such that
    \begin{equation}
        \|x^\alpha f\|_{L^\infty}\lesssim A^{|\alpha|}(\alpha!)^r \quad \text{ and } \quad \|\xi^\beta\widehat{f}\|_{L^\infty}\lesssim B^{|\beta|}(\beta!)^s, \ \quad \alpha,\beta\in\bN^d.
    \end{equation}
    \item There exist constants $A,B>0$ such that
    \begin{equation}
        \|x^\alpha f\|_{L^\infty}\lesssim A^{|\alpha|}(\alpha!)^r \quad \text{ and } \quad \|\partial^\beta f\|_{L^\infty}\lesssim B^{|\beta|}(\beta!)^s, \ \quad \alpha,\beta\in\bN^d.
    \end{equation}
    \item There exist constants $h,k>0$ such that
    \begin{equation}
        \|fe^{h|x|^{1/r}}\|_{L^\infty}<\infty \quad \text{ and } \quad \|\widehat{f}e^{k|\xi|^{1/s}}\|_{L^\infty}<\infty.
    \end{equation}
\end{enumerate}
\end{theorem}

A suitable window class for weighted modulation spaces is the Gelfand-Shilov-type space $\Sigma^1_1(\rd)$, consisting of functions $f\in\cS(\rd)$ such that, for every constant $A,B>0$,
\begin{equation}
|x^\alpha\partial^\beta f(x)|\lesssim A^{|\alpha|}B^{|\beta|}\alpha!\beta!, \quad \alpha,\beta\in\bN^d.
\end{equation}
We also have the inclusions $\cS^s_s(\rd)\subseteq\Sigma^1_1(\rd)\subseteq\cS^1_1(\rd)$ for every $s<1$. The characterization given in Theorem \ref{thm:Gelfand-Shilov-characterization} can be adapted to $\Sigma^1_1(\rd)$ by substituting \enquote{there exists} with \enquote{for every} and taking $r=s=1$.

The strong dual spaces of $\cS^s_r(\rd)$ and $\Sigma^1_1(\rd)$ are called spaces of tempered ultradistributions and denoted by $(\cS^s_r)'(\rd)$ and $(\Sigma^1_1)'(\rd)$, respectively. Notice that they contain the space of tempered distributions $\cS'(\rd)$.

The spaces $\cS^s_r(\rd)$ are important examples of nuclear spaces, and this property yields a kernel theorem for Gelfand-Shilov spaces.

\begin{theorem}
\label{thm:kernel-theorem-gelfand-shilov}
Let $r,s>0$ with $r+s\geq1$. There is a kernel correspondence between the continuous linear maps $T:\cS^s_r(\rd)\to(\cS^s_r)'(\rd)$ and $(\cS^s_r)'(\rdd)$, which associates to every $T$ a kernel $K_T\in(\cS^s_r)'(\rdd)$ such that
\begin{equation}
\scal{Tu}{v}=\scal{K_T}{v\otimes\overline{u}}, \quad \text{ for all } u,v\in\cS^s_r(\rd).
\end{equation}
Here $K_T(x,y)$ is ordered with the output variable first. The formula uses the duality convention fixed in the introduction, linear in the distribution and antilinear in the test function.
\end{theorem}

The following characterization expresses Gelfand-Shilov regularity directly as decay in phase space.

\begin{theorem}[Gelfand-Shilov characterization]
Let $s\geq1/2$ and let $g\in\cS^s_s(\rd)\setminus\{0\}$ be a Gelfand-Shilov window.
\begin{enumerate}[(i)]
    \item If $f\in\cS^s_s(\rd)$, then $V_gf\in\cS^s_s(\rdd)$.
    \item If $f\in\cS(\rd)$, then $f\in\cS^s_s(\rd)$ if and only if there exists $\varepsilon>0$ such that $|V_gf(z)|\lesssim e^{-\varepsilon|z|^{1/s}}$ for all $z\in\rdd$.
\end{enumerate}
\end{theorem}
For the Beurling space $\Sigma^s_s$, with $s>1/2$ and $0\neq g\in\Sigma^s_s$, the corresponding estimate must hold for every $\varepsilon>0$, with a constant depending on $\varepsilon$. The quantifier is part of the distinction between the two scales.

\subsection{Modulation spaces as projective and inductive limits}

We are now ready to extend modulation spaces to the setting of ultradistributions as follows. We now allow exponential submultiplicative control weights, as well as the polynomial and subexponential ones considered earlier.

\begin{definition}
Let $1\leq p,q\leq\infty$, let $m$ be a $v$-moderate weight on $\rdd$, and fix $g\in\Sigma^1_1(\rd)\setminus\{0\}$. The ultra-modulation space $M_m^{p,q}(\rd)$ consists of all tempered ultradistributions $f\in(\Sigma^1_1)'(\rd)$ such that
\begin{equation}
\label{eq:ultra-modulation-space-stft}
 \norm{f}_{M_m^{p,q}}:=\norm{V_gf}_{L_m^{p,q}}=\left(\int_{\rd}\left(\int_{\rd}|V_gf(x,\xi)|^pm(x,\xi)^pdx\right)^{q/p}d\xi\right)^{1/q}<\infty,
\end{equation}
with the usual endpoint modifications and the same notation conventions as above.
\end{definition}

Notice that this definition is formally the same as Definition~\ref{def:modulation-space-stft}, but with the window in $\Sigma^1_1(\rd)$ and the signal in $(\Sigma^1_1)'(\rd)$. For polynomial weights this recovers the previous spaces; for faster moderate weights it supplies the appropriate ambient dual.

We also observe that the above integral is convergent for $f,g\in\Sigma^1_1(\rd)$, and thus $\Sigma^1_1(\rd)\subseteq M_m^{p,q}(\rd)$, that is, the modulation space contains the space of windows, as happens in the case of weights of polynomial growth.

The basic properties of window independence and completeness hold in this more general setting.

\begin{theorem}[Window independence]
\label{thm:window-independence-ultra-modulation}
Let $1\leq p,q\leq\infty$ and let $m$ be a $v$-moderate weight. If $g_1,g_2\in\Sigma^1_1(\rd)\setminus\{0\}$, then
\begin{equation}
 \norm{V_{g_1}f}_{L_m^{p,q}}\asympconst \norm{V_{g_2}f}_{L_m^{p,q}},\qquad f\in(\Sigma^1_1)'(\rd).
\end{equation}
\end{theorem}

\begin{theorem}[Completeness and density]
Let $1\leq p,q\leq\infty$ and let $m$ be a $v$-moderate weight. Then the space $M_m^{p,q}(\rd)$ is a Banach space, and for $1\leq p,q<\infty$, $\Sigma^1_1(\rd)$ is a dense subspace of $M_m^{p,q}(\rd)$.
\end{theorem}

A different point of view about Gelfand-Shilov spaces and ultradistributions spaces is given in \cite{PilipovicTeofanov2002}, where the authors characterize them as projective and inductive limits of weighted modulation spaces.

Fix $\gamma\in(0,1)$ and consider a weight function of exp-type, i.e., a weight function $m$ such that
\begin{equation}
m(x+y,\xi+\eta)\lesssim e^{s(|x|^\gamma+|\xi|^\gamma)}m(y,\eta), \ \quad x,y,\xi,\eta\in\rd,
\end{equation}
for some $s\geq0$. Notice that $m$ is a weight function of exp-type if it is moderate with respect to the weight $v(x,\xi)=e^{s(|x|^\gamma+|\xi|^\gamma)}$ for some $s\geq0$. Particular instances of weight functions of exp-type are the weights
\begin{equation}
    m(x,\xi)=(1+|x|+|\xi|)^ae^{b|x|^{\gamma_1}+c|\xi|^{\gamma_2}}, \ \quad x,\xi\in\rd,
\end{equation}
for $a,b,c\geq0$ and $0<\gamma_1,\gamma_2\leq\gamma<1$; a zero exponent contributes only a constant factor.

Note also that an exp-type weight $m$ satisfies Beurling-Domar's non-quasianalyticity condition:
\begin{equation}
\sum_{n=1}^{+\infty}n^{-2}\log(m(nx,n\xi))<\infty, \ \quad x,\xi\in\rd.
\end{equation}

For fixed $\gamma\in(0,1)$ and $h>0$, we can consider the space $\cS^{(\gamma)}_h(\rd)$ of smooth functions on $\rd$ such that
\begin{equation}
    \sup_{\alpha,\beta\in\bN^d}\left(\frac{h^{|\alpha|+|\beta|}}{(\alpha!)^{1/\gamma}(\beta!)^{1/\gamma}}\|x^\alpha\partial^\beta f\|_\infty\right)<\infty.
\end{equation}
In the notation of Definition~\ref{gelfand-shilov}, the space $\cS^{(\gamma)}_h(\rd)$ coincides with the Gelfand-Shilov-type space $\cS^{s,A}_{r,B}(\rd)$, for $s=r=1/\gamma$ and $A=B=1/h$. Its projective limit is
 \begin{equation}
    \cS^{(\gamma)}(\rd):=\projlim_{h\to\infty}\cS^{(\gamma)}_h(\rd),
\end{equation}
and it naturally coincides with the space $\Sigma^s_r(\rd)$, for $s=r=1/\gamma$. Its topological dual $(\cS^{(\gamma)})'(\rd)=(\Sigma^{1/\gamma}_{1/\gamma})'(\rd)$ is called the space of Gevrey-Beurling tempered ultradistributions.

If we now consider the ultra-modulation space $M_m^{p,q}(\rd)$ as the set of all tempered ultradistributions $f\in(\Sigma^{1/\gamma}_{1/\gamma})'(\rd)=(\cS^{(\gamma)})'(\rd)$, for fixed $\gamma\in(0,1)$, $1\leq p,q\leq\infty$, $m$ an exp-type weight with this exponent $\gamma$, and $g\in\Sigma^{1/\gamma}_{1/\gamma}(\rd)\setminus\{0\}=\cS^{(\gamma)}(\rd)\setminus\{0\}$, such that
\begin{equation}
 \norm{f}_{M_m^{p,q}}:=\norm{V_gf}_{L_m^{p,q}}<\infty,
\end{equation}
we obtain the following projective-limit characterization:
\begin{equation}
    \cS^{(\gamma)}(\rd)=\Sigma^{1/\gamma}_{1/\gamma}(\rd)=\projlim_{s\to\infty}M_{m_s}^{2,2}(\rd),
\end{equation}
where $m_s(x,\xi)=e^{s(|x|^\gamma+|\xi|^\gamma)}$. As a consequence of this relation, by taking the dual spaces of both sides and using the projective-inductive duality valid for this Gelfand-Shilov scale, we obtain the following relation between spaces of ultradistributions and inductive limits of weighted modulation spaces:
\begin{equation}
    (\cS^{(\gamma)})'(\rd)=(\Sigma^{1/\gamma}_{1/\gamma})'(\rd)=\indlim_{s\to\infty}M_{1/m_s}^{2,2}(\rd).
\end{equation}

\subsection{Pseudodifferential operators on Gelfand-Shilov spaces}

The same enlargement of test and distribution spaces also extends the kernel and symbol calculus of pseudodifferential operators. The phase convention is unchanged from the Weyl quantization in Section~\ref{sec:gabor-operators}.

Let $s\geq1/2$ and $\tau\in\bR$. The Shubin $\tau$-representation of a pseudodifferential operator with symbol $\sigma\in\cS^s_s(\rdd)$ is the linear operator $\Op_\tau(\sigma)$ given by:
\begin{equation}
\Op_\tau(\sigma)f(x)=\int_{\rdd}\sigma((1-\tau)x+\tau y,\xi)e^{2\pi i(x-y)\cdot\xi}f(y)dyd\xi.
\end{equation}
Notice that, for $\tau=1/2$, we recover the Weyl quantization defined in \eqref{eq:Weyl-quantization}, that is $\Opw(\sigma)=\Op_{1/2}(\sigma)$. For $\tau=0$, we obtain the Kohn-Nirenberg operator $\sigma(x,D)$ with symbol $\sigma$, that is
\begin{equation}
    \sigma(x,D)f(x)=\int_{\rdd}\sigma(x,\xi)e^{2\pi i(x-y)\cdot\xi}f(y)dyd\xi=\int_{\rd}\sigma(x,\xi)e^{2\pi ix\cdot\xi}\widehat{f}(\xi)d\xi.
\end{equation}
For a recent group-theoretic treatment of $\tau$-pseudodifferential operators on locally compact abelian groups, with modulation spaces as the underlying functional framework, see \cite{Yaremenko2026Gabor}.

The definition of Shubin $\tau$-representation can be extended in weak form to include the case of symbols $\sigma\in(\cS^s_s)'(\rdd)$. Indeed, in this case, $\Op_\tau(\sigma)$ is the linear and continuous operator $\Op_\tau(\sigma):\cS^s_s(\rd)\to(\cS^s_s)'(\rd)$ having distributional kernel
\begin{equation}
    K_{\sigma,\tau}(x,y)=(\cF_2^{-1}\sigma)((1-\tau)x+\tau y,x-y), \quad x,y\in\rd,
\end{equation}
where $\cF_2$ denotes the Fourier transform with respect to the second variable. This definition makes sense because the linear mappings
\begin{equation}
    \cF_2 \quad \text{ and } \quad F(x,y)\mapsto F((1-\tau)x+\tau y,x-y),
\end{equation}
are topological isomorphisms on $(\cS^s_s)'(\rdd)$. In particular, the map $\sigma\mapsto K_{\sigma,\tau}$ is a homeomorphism on $(\cS^s_s)'(\rdd)$.

By the inversion formula and the kernel theorem for operators from Gelfand-Shilov spaces to their duals (Theorem \ref{thm:kernel-theorem-gelfand-shilov}), it follows that the map $\sigma\mapsto\Op_\tau(\sigma)$ is a bijection from $(\cS^s_s)'(\rdd)$ to the set of all linear and continuous operators from $\cS^s_s(\rd)$ to $(\cS^s_s)'(\rd)$.

Shubin representations for different choices of $\tau$ are related in the following sense: if $\tau_1,\tau_2\in\bR$ and $\sigma_1\in(\cS^s_s)'(\rdd)$, there exists a unique $\sigma_2\in(\cS^s_s)'(\rdd)$ such that $\Op_{\tau_1}(\sigma_1)=\Op_{\tau_2}(\sigma_2)$.

\section{Frequency-uniform decompositions and a PDE application}
\label{sec:guo-baoxiang}

We now return to the frequency-uniform decomposition underlying Feichtinger's original construction. Whereas Gabor coefficients discretize both time and frequency, this formulation keeps the spatial norm continuous and discretizes frequency. This is particularly convenient for dispersive estimates: one can first exploit the action of the propagator on each frequency cell and then sum the resulting bounds. We retain the polynomial control-weight assumptions of Sections~\ref{sec:stft-definition} and~\ref{sec:gabor}.

\subsection{Frequency-uniform decompositions}

To apply the characterization from Section~\ref{sec:feichtinger-original}, we choose an explicit smooth partition adapted to the integer lattice.

For $k\in\zd$, we denote by $\mathcal{Q}_k$ the closed unit cube of $\rd$ centered at $k$. Notice that two different unit cubes overlap on a null-measure set and the family $\{\mathcal{Q}_k\}_{k\in\zd}$ is a covering of $\rd$. We define $|\xi|_\infty:=\max_{1\leq i\leq d}|\xi_i|$.

Consider a smooth function $\rho:\rd\to[0,1]$ satisfying the conditions
\begin{equation}
    \rho(\xi)=1 \quad \text{ for } |\xi|_\infty\leq1/2, \quad \rho(\xi)=0 \quad \text{ for } |\xi|_\infty\geq3/4.
\end{equation}
Define
\begin{equation}
    \rho_k(\xi)=T_k\rho(\xi)=\rho(\xi-k), \quad \text{ for } k\in\zd,
\end{equation}
and notice that $\rho_k$ is the translation of $\rho$ at $k$. By the assumption on $\rho$, we obtain that $\rho_k(\xi)=1$ for $\xi\in\mathcal{Q}_k$ and
\begin{equation}
    \sum_{k\in\zd}\rho_k(\xi)\geq1, \quad \text{ for } \xi\in\rd.
\end{equation}
Also denote by
\begin{equation}
    \sigma_k(\xi)=\frac{\rho_k(\xi)}{\sum_{l\in\zd}\rho_l(\xi)}, \quad \text{ for } \xi\in\rd,k\in\zd.
\end{equation}
Observe that $\sigma_k(\xi)=\sigma_0(\xi-k)\in\cD(\rd)$ and the sequence $\{\sigma_k\}_{k\in\zd}$ is a smooth partition of unity adapted to uniformly enlarged unit cubes in frequency space, i.e.,
\begin{equation}
    \sum_{k\in\zd}\sigma_k(\xi)=1, \quad \text{ for } \xi\in\rd.
\end{equation}
This smooth partition of unity allows us to define the frequency-uniform decomposition operators.

\begin{definition}
    The frequency-uniform decomposition operators are defined by
    \begin{equation}
        \Box_k=\cF^{-1}\sigma_k\cF, \quad \text{ for } k\in\zd.
    \end{equation}
\end{definition}

The previous operators allow us to introduce a discrete norm on the weighted modulation spaces $M_{m_1\otimes m_2}^{p,q}(\rd)$ as follows.
\begin{definition}
    Let $1\leq p,q\leq\infty$ and let $m_1,m_2$ be two $v$-moderate weights. A discrete norm for modulation spaces is defined as
    \begin{equation} \label{eq:uniform-decomposition}
        \norm{f}_{M_{m_1\otimes m_2}^{p,q},\mathrm{dec}}:=\left(\sum_{k\in\zd}\norm{\Box_k f}_{L^p_{m_1}}^qm_2(k)^q\right)^{1/q}, \quad \text{ for } f\in\cS'(\rd),
    \end{equation}
    with the usual modification for $q=\infty$.
\end{definition}

This expression recovers the STFT norm:
\begin{theorem}
    Let $1\leq p,q\leq\infty$ and let $m_1,m_2$ be two $v$-moderate weights. Then the norm \eqref{eq:uniform-decomposition} is an equivalent norm on $M_{m_1\otimes m_2}^{p,q}(\rd)$, and it is independent of the choice of $\sigma$.
\end{theorem}

\subsection{A representative application to nonlinear Schr\"odinger equations}

A useful illustration is provided by dispersive equations with rough initial data. The unit-cube decomposition in frequency permits one to exploit the local behavior of the Schr\"odinger propagator while measuring the global distribution of the localized pieces in a modulation norm. We recall one such result from \cite{WangHudzik2007,WangZhaoGuo2011}.

Consider the nonlinear Schr\"odinger equation
\begin{equation} \label{eq:NLS}
    \begin{cases}
        iu_t+\Delta u+f(u)=0, \\
        u(0)=u_0.
    \end{cases}
\end{equation}
The classical dispersive estimate is
\begin{equation} \label{eq:estimate-S-Lp}
    \|e^{it\Delta}\varphi\|_{L^p}\lesssim|t|^{-d(1/2-1/p)}\|\varphi\|_{L^{p'}}, \quad 2\leq p\leq\infty.
\end{equation}
Its right-hand side is singular at $t=0$. After frequency-uniform localization, the corresponding modulation-space estimate takes the form

\begin{equation} \label{eq:estimate-S-Mp}
    \|e^{it\Delta}f\|_{M_{1\otimes v_s}^{p,q}}\lesssim(1+|t|)^{-d(1/2-1/p)}\|f\|_{M_{1\otimes v_s}^{p',q}},
\end{equation}

with $2\leq p\leq\infty$, $0<q\leq\infty$, and $s\in\bR$. Thus the decay rate for large $|t|$ is preserved while the singularity at $t=0$ disappears.

\begin{theorem} \label{thm:solution-NLS}
    Let $k_0$ be the positive root of $dk^2+(d-2)k-4=0$, let $k\in\bN$ with $k>k_0$, and let $f(u)$ be a monomial of total degree $k+1$ in $u$ and $\overline u$. Set
    \begin{equation}
        p=\frac{k+2}{k+1}.
    \end{equation}
    There exists a sufficiently small $\delta>0$ such that, for every $u_0\in M^{p,1}(\rd)$ with $\|u_0\|_{M^{p,1}}\leq\delta$, the Cauchy problem \eqref{eq:NLS} has a unique global solution satisfying
    \begin{equation}
        \sup_{t\in\bR}(1+|t|)^{dk/(4+2k)}\|u(t)\|_{M^{2+k,1}}<\infty.
    \end{equation}
\end{theorem}

The point of this example for the present chapter is not the particular exponent range. It is the role of the modulation-space geometry. The same uniform localization that distinguishes modulation spaces from dyadic scales also regularizes the short-time behavior of the dispersive estimate. The much broader theory of sharp convolution, product, embedding, and PDE estimates is documented in the references cited above and in the literature surveyed in this chapter.

\section{Symplectic analysis of time-frequency spaces}
\label{sec:symplectic}

 {We return to polynomially moderate weights and the pair $(\cS,\cS')$ to address the characterization problem stated in the Introduction.} The symplectic matrices describe the geometry of the change; the order of the mixed norm determines which deformations are admissible.

\subsection{Metaplectic Wigner distributions}

 {For time-frequency distributions on $\rdd$ we work with $\Mp(2d,\bR)$ and its projection onto $\Sp(2d,\bR)$, using for $\mathcal A$ the $d\times d$ block decomposition already fixed in \eqref{intro.blockA}.}
 {We use the metaplectic Wigner distribution $W_{\mathcal A}$ defined in the Introduction. When $f=g$, we write $W_{\mathcal A}f=W_{\mathcal A}(f,f)$.}

Two questions about $W_{\mathcal A}$ should be kept distinct. One concerns the regularity of the distribution $W_{\mathcal A}(f,g)$ itself. The other, which is central for modulation spaces, asks whether a norm of $W_{\mathcal A}(f,g)$ recovers the norm of the input $f$. The continuity properties and Moyal's identity below provide the general calculus. The later shift-invertibility condition is what turns this calculus into a localization theory comparable with the STFT.

This construction includes the principal bilinear representations used here. We first record its continuity and Moyal's identity, and then identify the STFT, Wigner and ambiguity examples.

The next result gives the continuity of the metaplectic Wigner distribution in $L^2$ spaces, in the Schwartz class, and in the space of tempered distributions.

\begin{theorem} \label{thm:action-metaplectic-wigner-distribution}
    Let $\mathcal A\in\Sp(2d,\bR)$.
    \begin{enumerate}[(i)]
        \item If $f,g\in L^2(\rd)$, then $W_{\mathcal A}(f,g)\in L^2(\rdd)$ and the map $W_{\mathcal A}:L^2(\rd)\times L^2(\rd)\to L^2(\rdd)$ is continuous.
        \item If $f,g\in\cS(\rd)$, then $W_{\mathcal A}(f,g)\in\cS(\rdd)$ and the map $W_{\mathcal A}:\cS(\rd)\times\cS(\rd)\to\cS(\rdd)$ is continuous.
        \item If $f,g\in\cS'(\rd)$, then $W_{\mathcal A}(f,g)\in\cS'(\rdd)$ and the map $W_{\mathcal A}:\cS'(\rd)\times\cS'(\rd)\to\cS'(\rdd)$ is continuous.
    \end{enumerate}
\end{theorem}

The classical Moyal's identity which holds for the classical Wigner distribution can be extended to the metaplectic Wigner distribution.

\begin{theorem}[Moyal's identity]
    Let ${\mathcal A}\in\Sp(2d,\bR)$, and $f_1,f_2,g_1,g_2\in L^2(\rd)$. Then
    \begin{equation}
        \scal{{W_{\mathcal A}}(f_1,g_1)}{{W_{\mathcal A}}(f_2,g_2)}_{L^2}=\scal{f_1}{f_2}_{L^2}\overline{\scal{g_1}{g_2}_{L^2}}.
    \end{equation}
    In particular, for $f_1=f_2=f$ and $g_1=g_2=g$, we have
    \begin{equation}
        \norm{W_{\mathcal A}(f,g)}^2_{L^2}=\norm{f}^2_{L^2}\norm{g}^2_{L^2}.
    \end{equation}
\end{theorem}

\subsection{Classical examples}

The following examples connect the abstract definition with the familiar analyzers. The displayed formulas use the standard time-frequency normalization; their metaplectic lifts are understood with the constant-phase convention stated above.

\begin{example}[STFT]
Consider the symplectic matrix in block form
\begin{equation} \label{eq:matrix-STFT}
    \cA_{st}=
    \begin{pmatrix}
        I_d & -I_d & 0_d & 0_d \\
        0_d & 0_d & I_d & I_d \\
        0_d & 0_d & 0_d & -I_d \\
        -I_d & 0_d & 0_d & 0_d
    \end{pmatrix}\in\Sp(2d,\bR),
\end{equation}
and its metaplectic representation ${\hat\cA_{st}}\in\Mp(2d,\bR)$. Then
\begin{equation}
 {\hat\cA_{st}}(f\otimes\overline{g})=V_gf, \qquad f,g\in L^2(\rd).
\end{equation}
Therefore, for $\mathcal A=\cA_{st}$, we recover the STFT.
\end{example}

\begin{example}[$\tau$-Wigner distributions]
Let $\tau\in\bR$, and consider the symplectic matrix in block form
\begin{equation} \label{eq:matrix-tau-wigner}
    \mathcal A_{\tau}=
    \begin{pmatrix}
        (1-\tau)I_d & \tau I_d & 0_d & 0_d \\
        0_d & 0_d & \tau I_d & -(1-\tau)I_d \\
        0_d & 0_d & I_d & I_d \\
        -I_d & I_d & 0_d & 0_d
    \end{pmatrix}\in\Sp(2d,\bR),
\end{equation}
and its metaplectic representation $\widehat{\mathcal A}_{\tau}\in\Mp(2d,\bR)$.  {For $\mathcal A=\mathcal A_{\tau}$ we recover precisely the $\tau$-Wigner distribution defined in \eqref{intro.def.Wtau}.} In particular, we retrieve the classical Wigner distribution for $\tau=1/2$, i.e.,
\begin{equation}
    W(f,g)(x,\xi)=\int_{\rd}f(x+y/2)\overline{g(x-y/2)}e^{-2\pi i\xi\cdot y}\ud y, \qquad f,g\in L^2(\rd), \, x,\xi\in\rd
\end{equation}
and 
\begin{equation}
    \mathcal A_{1/2}=
    \begin{pmatrix}
        \frac{1}{2}I_d & \frac{1}{2}I_d & 0_d & 0_d \\
        0_d & 0_d & \frac{1}{2}I_d & -\frac{1}{2}I_d \\
        0_d & 0_d & I_d & I_d \\
        -I_d & I_d & 0_d & 0_d
    \end{pmatrix}.
\end{equation}
\end{example}

\begin{example}[Ambiguity function]
 {Consider the partial Fourier transform with respect to the last $d$ variables, already defined up to a constant phase in \eqref{intro.defF2}.} It extends unitarily to $L^2(\rdd)$ and, up to a constant phase, is metaplectic. Its symplectic projection is
\begin{equation}
    \mathcal A_{\mathrm{FT2}}=
    \begin{pmatrix}
        I_d & 0_d & 0_d & 0_d \\
        0_d & 0_d & 0_d & I_d \\
        0_d & 0_d & I_d & 0_d \\
        0_d & -I_d & 0_d & 0_d
    \end{pmatrix}\in\Sp(2d,\bR).
\end{equation}
Composing $\cF_2$ with the unimodular rescaling $\mathfrak T_{E_{\mathrm{amb}}}$, $E_{\mathrm{amb}}=\left(\begin{smallmatrix}\frac12 I_d& I_d\\ -\frac12 I_d& I_d\end{smallmatrix}\right)$, one obtains the cross-ambiguity function
\begin{equation}
 \operatorname{Amb}(f,g)(x,\xi)=\cF_2\mathfrak T_{E_{\mathrm{amb}}}(f\otimes\bar g)(x,\xi)
 =\int_{\rd}f(y+x/2)\overline{g(y-x/2)}e^{-2\pi i y\cdot\xi}\ud y, \qquad f,g\in L^2(\rd),\; x,\xi\in\rd.
\end{equation}
The change of variables $y'=y+x/2$ shows that $\operatorname{Amb}(f,g)(x,\xi)=e^{\pi i x\cdot\xi}V_gf(x,\xi)$. Thus it is obtained from the STFT by a quadratic chirp on phase space.
\end{example}

\subsection{Equivalence of modulation-space norms}

The preceding examples show that the metaplectic family is algebraically broad.  {We now return to the characterization problem formulated in \eqref{intro.equiv.mod.spaces}: which members of this family may replace the STFT as genuine phase-space analyzers without changing the underlying modulation space?}

The classical Wigner distribution 
already illustrates both the possibility of changing analyzer and the need to transform the weight. The relation
\begin{equation}\label{eq:wigner-stft-rescaling}
 W(f,g)(x,\xi)=2^d e^{4\pi i x\cdot\xi}V_{\cI g}f(2x,2\xi),
 \qquad \cI g(t)=g(-t),
\end{equation}
leads, for $0<p,q\leq\infty$ and $0\neq g\in\cS(\rd)$, to
\begin{equation}
 \|f\|_{M_m^{p,q}}\asymp\|W(f,g)\|_{L_{m\circ(2I_{2d})}^{p,q}}.
\end{equation}
If $m(2z)\asymp m(z)$, as for radial polynomial weights, the weight on the right may be written simply as $m$. 

For $\tau\in(0,1)$, the analogous statement is
\begin{equation}
 \|f\|_{M_m^{p,q}}\asymp\|W_\tau(f,g)\|_{L_{m\circ E_\tau^{-1}}^{p,q}},
\end{equation}
where
\begin{equation}\label{defEtau}
 E_\tau=\begin{pmatrix}
 (1-\tau)I_d&0_d\\0_d&\tau I_d
 \end{pmatrix}.
\end{equation}
The associated window is reflected and rescaled. In the Banach range the Schwartz assumption can be relaxed when this transformed window belongs to $M_v^1$.

 {For a general symplectic matrix $\mathcal A\in\Sp(2d,\bR)$, the answer to \eqref{intro.equiv.mod.spaces} relies on the mapping properties of metaplectic operators on weighted mixed-norm modulation spaces.}
Here the distinction between the one-sided weight $1\otimes v_s$ and $v_s$ used radially on phase space is essential: a general symplectic transformation mixes space and frequency, so an arbitrary metaplectic operator cannot preserve a frequency-only weight.

\begin{theorem}[Sharp metaplectic action on $M^{p,q}$]\label{thm:sharp-metaplectic-action}
 Let $1\leq p,q\leq\infty$ and let $\widehat \delta\in\Mp(d,\bR)$ have symplectic projection 
 \begin{equation}
     \delta=\begin{pmatrix}
         A & B\\
         C & D
     \end{pmatrix}.
 \end{equation}
 Then the following are equivalent:
\begin{enumerate}[(i)]
 \item $\widehat \delta:M^{p,q}(\rd)\to M^{p,q}(\rd)$ is well defined.
 \item $\widehat \delta:M^{p,q}(\rd)\to M^{p,q}(\rd)$ is bounded (indeed, an automorphism).
 \item either $p=q$, or $p\neq q$ and $\delta$ is upper block triangular, equivalently $C=O_d$.
\end{enumerate}
If $m$ is a polynomially bounded moderate weight satisfying
\begin{equation}\label{eq:weight-metaplectic-compatibility}
 m\asymp m\circ \delta^{-1},
\end{equation}
then the corresponding boundedness result transfers to $M_m^{p,q}(\rd)$.
 \end{theorem}
 \noindent This is the sharp classification of
\cite{FuhrShafkulovska2024}. See also the weighted and quasi-Banach discussion in
\cite{CorderoGiacchi2024}.  Since $v_s(\delta z)\asymp v_s(z)$ for every fixed invertible linear map $\delta$, condition~\eqref{eq:weight-metaplectic-compatibility} is automatic for radial polynomial phase-space weights.  Consequently every metaplectic operator acts continuously on the diagonal spaces $M_{v_s}^{p}$ and, in the quasi-Banach diagonal regime, the analogous statement follows from the extensions summarized in
\cite{CorderoGiacchi2024}.  By contrast, no such unconditional invariance should be claimed for $M_{1\otimes v_s}^{p,q}$.

\subsection{Shift-invertibility and triangularity}

 {We use the shift-invertibility condition introduced in Definition~\ref{intro.defSI}. With the block notation \eqref{intro.blockA}, the associated matrix $E_{\mathcal A}$ is as in \eqref{intro.defEA} and $W_{\mathcal A}$ is shift-invertible if and only if $E_{\mathcal A}$ is invertible.}
 {
Together with $E_\cA$ defined in \eqref{intro.defEA}, we identify three other submatrices of $\cA$:
\begin{equation}
    \cE_\cA=\begin{pmatrix}
        A_{12} & A_{14}\\
        A_{22} & A_{24}
    \end{pmatrix}, \qquad 
    F_\cA=\begin{pmatrix}
        A_{31} & A_{33}\\
        A_{41} & A_{43}
    \end{pmatrix}, \qquad \cF_\cA=\begin{pmatrix}
        A_{32}& A_{34}\\
        A_{42} & A_{44}
    \end{pmatrix}.
\end{equation}
The invariance relation $\cA^\top J\cA=J$ defining the symplectic group reads, in terms of $E_\cA$, $\cE_\cA$, $F_\cA$ and $\cF_\cA$ as
\begin{equation}
    \begin{cases}
        E_\cA^\top F_\cA-F_\cA^\top E_\cA=J,\\
        \cE_\cA^\top \cF_\cA-\cF_\cA^\top\cE_\cA=J,\\
        E_\cA^\top\cF_\cA-F_\cA^\top\cE_\cA=O_{2d},
    \end{cases}
\end{equation}
where, with a slight abuse of notation, $J$ here denotes the canonical symplectic matrix of $\rdd$, as shown in \cite{CorderoGiacchi2024MetaplecticGabor}. Moreover, Lemma~2.5 therein relates $E_\cA$ and $\cE_\cA$ as follows.
\begin{theorem}\label{thmGA}
Let
\begin{equation}
    L=\begin{pmatrix}
        -I_d & O_d\\
        O_d & I_d
    \end{pmatrix}.
\end{equation}
    For $\cA\in\Sp(2d,\bR)$, we have that $\det(E_\cA)=(-1)^d\det(\cE_\cA)$. Moreover, if $E_\cA\in\mathrm{GL}(2d,\bR)$, then $\delta_\cA:=E_\cA^{-1}\cE_\cA L\in\Sp(d,\bR)$.
\end{theorem}
}
 {We now record the basic examples needed below.}

\begin{example}[STFT]
For the short-time Fourier transform, the matrix $\mathcal A=\cA_{st}$ is shown in \eqref{eq:matrix-STFT}. The associated matrix $E_{st}:=E_{\cA_{st}}$ is given by
\begin{equation}
    E_{st}=
    \begin{pmatrix}
        I_d & 0_d \\
        0_d & I_d
    \end{pmatrix},
\end{equation}
so that $\cA_{st}$ is shift-invertible. Equivalently, it suffices to recall that
\begin{equation}
    |V_g\pi(w)f|=|T_wV_gf|,  \quad f,g\in L^2(\rd), \, w\in\rdd.
\end{equation}
\end{example}

\begin{example}[$\tau$-Wigner distributions]
    For $\tau$-Wigner distributions, with $\tau\in\bR$, the matrix $\mathcal A=\mathcal A_{\tau}$ is shown in \eqref{eq:matrix-tau-wigner}. The associated matrix $E_{\tau}:=E_{\mathcal A_{\tau}}$ is given by
    \eqref{defEtau} so that $\mathcal A_{\tau}$ is shift-invertible if and only if $\tau\in\bR\setminus\{0,1\}$.
\end{example}

The importance of shift-invertibility is not merely that $E_{\mathcal A}$ can be inverted. It identifies precisely the metaplectic Wigner distributions that are STFTs in deformed coordinates.

\begin{theorem}[Shift-invertibility as a rescaled-STFT principle]\label{thm:rescaled-stft}
Let $W_{\mathcal A}$ be shift-invertible. Consider the symmetric matrix
\begin{equation}
    N_\cA=\begin{pmatrix}
        A_{11}^\top A_{31}+A_{21}^\top A_{41} & A_{11}^\top A_{33}+A_{21}^\top A_{43}\\
        A_{13}^\top A_{31}+A_{23}^\top A_{41}+I_d & A_{13}^\top A_{33}+A_{23}^\top A_{43}
    \end{pmatrix}
\end{equation}
and the metaplectic operator $\widehat{\delta_\cA}\in\Mp(d,\bR)$ associated to $\delta_\cA$, defined in Theorem \ref{thmGA}.
Then, up to an overall sign,
\begin{equation}\label{eq:rescaled-stft}
 W_{\mathcal A}(f,g)(z)
 =|\det(E_{\mathcal A})|^{-1/2}
 \Phi_{N_{\mathcal A}}(E_{\mathcal A}^{-1}z)
 V_{\widehat{\delta_{\mathcal A}}g}f(E_{\mathcal A}^{-1}z).
\end{equation}
\end{theorem}
The operator $\widehat{\delta_\cA}$ mentioned in the previous statement is called {\em deformation operator} associated with $W_\cA$.
This structural characterization \cite{CorderoGiacchi2024MetaplecticGabor} makes the modulation-space result transparent. The chirp has modulus one, so it does not alter the mixed norm. A Schwartz window remains Schwartz under the metaplectic deformation. What remains is the linear change of variables $E_{\mathcal A}^{-1}$. On diagonal spaces there is no ordering issue. For $M_m^{p,q}$ with $p\neq q$, however, the two integrations in the mixed norm play different roles, and upper-triangular changes preserve their order. For Wiener amalgam spaces the local-global order is reversed, which leads to the lower-triangular condition below.
More explicitly, if $E=\left(\begin{smallmatrix}U&V\\0&W\end{smallmatrix}\right)$ with $U,W\in\GL(d,\bR)$, then
\begin{equation}\label{eq:triangular-mixed-change}
 \|F\circ E^{-1}\|_{L_m^{p,q}}
 =|\det U|^{1/p}|\det W|^{1/q}\|F\|_{L_{m\circ E}^{p,q}},
\end{equation}
with $1/\infty=0$. At each fixed outer variable, the inner change is only an invertible affine change in $x$. This is why block triangularity, rather than arbitrary invertibility, matches a mixed norm.

We can now state the characterization supplied by shift-invertibility. For mixed modulation spaces, the additional upper block-triangular hypothesis on $E_{\mathcal A}$ ensures compatibility with the order of the mixed norm. No such ordering hypothesis is required on the diagonal.

  \begin{theorem}[Shift-invertible characterization \cite{CorderoGiacchi2023JMPA,CorderoGiacchi2024MetaplecticGabor}]  \label{thm:characterization-modulation-metaplectic-wigner}
    Let $g\in\cS(\rd)\setminus\{0\}$, $\mathcal A\in\Sp(2d,\bR)$ such that $W_{\mathcal A}$ is shift-invertible, and $m$ be a $v$-moderate weight such that $m\asympconst m\circ E_{\mathcal A}^{-1}$.
    \begin{enumerate}[(i)]
        \item For $0<p\leq\infty$, we have
        \begin{equation}
            f\in M^p_m(\rd) \quad\Longleftrightarrow\quad W_{\mathcal A}(f,g)\in L^p_m(\rdd),
        \end{equation}
        with equivalence of norms $\norm{f}_{M^p_m}\asympconst\norm{W_{\mathcal A}(f,g)}_{L^p_m}$.
        \item For $0<p,q\leq\infty$, if we add the assumption that $E_{\mathcal A}$ is upper block triangular, we have
        \begin{equation}
            f\in M_m^{p,q}(\rd) \quad\Longleftrightarrow\quad W_{\mathcal A}(f,g)\in L_m^{p,q}(\rdd),
        \end{equation}
        with equivalence of norms $\norm{f}_{M_m^{p,q}}\asympconst\norm{W_{\mathcal A}(f,g)}_{L_m^{p,q}}$.
    \end{enumerate}
\end{theorem}

For Banach exponents, the same statements hold for a nonzero $L^2$ window $g$ whenever the deformed window $\widehat{\delta_{\mathcal A}}g$ belongs to $M_v^1$. In particular, $g\in M_v^1$ suffices if the deformation preserves $M_v^1$, as happens when $v$ is equivalent to a radial polynomial weight.

\begin{remark}[The Hilbert-space case]
Shift-invertibility is a sufficient structural mechanism for these characterizations, not a necessary condition for every individual norm identity. When $p=q=2$ and $m=1$, Moyal's identity gives $\|W_{\mathcal A}(f,g)\|_2=\|f\|_2\|g\|_2$ for every metaplectic representation, including those that are not shift-invertible.
\end{remark}

In particular, Theorem~\ref{thm:characterization-modulation-metaplectic-wigner} covers the cases of STFT, $\tau$-Wigner distributions with $\tau\in\bR\setminus\{0,1\}$ and the ambiguity function.

For Wiener amalgam spaces, the local Fourier-Lebesgue norm is taken before the global spatial norm. This reverses the integration order and, correspondingly, the triangularity condition \cite[Theorem~7.2]{CorderoGiacchi2024MetaplecticGabor}.

\begin{theorem}[Wiener amalgam characterization]\label{thm:wiener-metaplectic}
Let $0<p,q\leq\infty$, $0\neq g\in\cS(\rd)$, and let $W_{\mathcal A}$ be shift-invertible. Let $m_1,m_2$ be polynomially moderate weights on $\rd$ such that $m_2(x)\asymp m_2(-x)$. Set
\begin{equation}
 \widetilde E_{\mathcal A}=JE_{\mathcal A}J,
\end{equation}
where $J$ is the canonical symplectic matrix of $\rdd$,
and assume
\begin{equation}
 m_1\otimes m_2\asymp(m_1\otimes m_2)\circ\widetilde E_{\mathcal A}^{-1}.
\end{equation}
If $E_{\mathcal A}$ is lower block triangular, then
\begin{equation}\label{eq:wiener-metaplectic-norm}
 \|f\|_{W(\cF L^p_{m_1},L^q_{m_2})}
 \asymp
 \left(\int_{\rd}\left(\int_{\rd}|W_{\mathcal A}(f,g)(x,\xi)|^p
 m_1(\xi)^p\,d\xi\right)^{q/p}m_2(x)^q\,dx\right)^{1/q},
\end{equation}
with the usual endpoint modifications. For $p=q$, the triangularity assumption may be omitted.
\end{theorem}
Here the inner variable is $\xi$ and the outer variable is $x$, unlike in $L_m^{p,q}$. The matrix $\widetilde E_{\mathcal A}$ is upper block triangular precisely when $E_{\mathcal A}$ is lower block triangular. This is the same change-of-variables principle as in \eqref{eq:triangular-mixed-change}, applied in the reversed order.

\begin{remark}[Covariance and the Cohen class]

The Cohen class gives a complementary classification of quadratic time-frequency representations. For Schwartz inputs, a sesquilinear representation $Q$ belongs to this class if
\begin{equation}
 Q(f,g)=k*W(f,g)
\end{equation}
for a fixed kernel $k\in\cS'(\rdd)$. Within the metaplectic family, covariance under simultaneous time-frequency shifts is equivalent to Cohen-class membership \cite{CorderoRodino2023JFA}. In particular, the Wigner and $\tau$-Wigner distributions are covariant and belong to the Cohen class, whereas the STFT does not. We do not pursue the explicit matrix classification here because it answers a different question from the one driving the present section. The characterization of modulation and Wiener amalgam spaces is governed instead by shift-invertibility together with the geometry of the mixed norm. The metaplectic family is useful precisely because it contains both Cohen-class distributions and analyzers, such as the STFT, that lie outside that subclass.

\end{remark}

\subsection{Atomic characterization from a metaplectic perspective}

Since the class of metaplectic Wigner distributions contains the STFT, we can use its general definition to introduce an implicit notion of metaplectic atom, which generalizes the time-frequency shifts used to define the STFT.

\begin{definition}
    Let $\mathcal A\in\Sp(2d,\bR)$ and $z\in\rdd$. For $f\in\cS(\rd)$, define the metaplectic atom $\pi_{\mathcal A}(z)f\in\cS'(\rd)$ by
    \begin{equation}
        \scal{\varphi}{\pi_{\mathcal A}(z)f}:=W_{\mathcal A}(\varphi,f)(z), \qquad \varphi\in\cS(\rd).
    \end{equation}
    Here $\langle\varphi,u\rangle=\overline{\langle u,\varphi\rangle}$ for $u\in\cS'$.
\end{definition}

The metaplectic atoms play the same role as the time-frequency shifts for the STFT. The next example shows that this definition is consistent with classical time-frequency shifts.

\begin{example}[STFT]
    For the short-time Fourier transform, projected onto the matrix $\mathcal A=\cA_{st}$ defined in \eqref{eq:matrix-STFT}, an explicit computation yields $\pi(z)=\pi_{\cA_{st}}(z)$. Indeed, for all $z\in\rdd$ and $\varphi,f\in\cS(\rd)$, we have $\scal{\varphi}{\pi(z)f}=V_f\varphi(z)=\scal{\varphi}{\pi_{\cA_{st}}(z)f}$, and therefore $\pi(z)f=\pi_{\cA_{st}}(z)f$ for all $f\in\cS(\rd)$.
\end{example}

\begin{example}[$\tau$-Wigner distributions]
    For the $\tau$-Wigner distribution, with $\tau\in(0,1)$, projected onto the matrix $\mathcal A=\mathcal A_{\tau}$ defined in \eqref{eq:matrix-tau-wigner}, an explicit computation yields
    \begin{equation}
        \pi_{\mathcal A_{\tau}}(x,\xi)f=|\det(E_{\tau})|^{-1/2}e^{-2\pi i\frac{x\cdot\xi}{\tau}}\pi(E_{\tau}^{-1}(x,\xi))\mathfrak{T}_{\tau}f,
    \end{equation}
    where $E_{\tau}$ is defined in \eqref{defEtau} and $\mathfrak{T}_{\tau}f(t)=\frac{(1-\tau)^{d/2}}{\tau^{d/2}}f\left(-\frac{1-\tau}{\tau}t\right)$.
\end{example}

\begin{example}
    The (cross)-Rihaczek distribution is the $\tau$-Wigner distribution with $\tau=0$, given by
    \begin{equation}
        W_0(f,g)(x,\xi)=f(x)\overline{\widehat{g}(\xi)}e^{-2\pi i x\cdot\xi}, \qquad (x,\xi)\in\rdd,
    \end{equation}
    and is projected onto the matrix $\mathcal A=\mathcal A_0$ defined in \eqref{eq:matrix-tau-wigner}. An explicit computation yields
    \begin{equation}
        \pi_{\mathcal A_0}(x,\xi)f=\widehat{f}(\xi)e^{2\pi i x\cdot\xi}T_x\delta_0,
    \end{equation}
    where $\delta_0$ is the delta distribution centered on $0$.
\end{example}

The last example shows that $\pi_{\mathcal A}(z)$ may not define a function.
\begin{theorem}
    Let $\mathcal A\in\Sp(2d,\bR)$, and $z\in\rdd$. Then the metaplectic atom $\pi_{\mathcal A}(z)$ defines a linear operator from $\cS(\rd)$ to $\cS'(\rd)$.
\end{theorem}

With the definition of metaplectic atom, we can easily find an inversion formula which extends the classical inversion formula valid for the Wigner distribution.
\begin{theorem}
    Let $\mathcal A\in\Sp(2d,\bR)$, $f,g\in L^2(\rd)$, and $\gamma\in\cS(\rd)$ such that $\scal{\gamma}{g}\neq0$. Then
    \begin{equation}
        f=\frac{1}{\scal{\gamma}{g}}\int_{\rdd}W_{\mathcal A}(f,g)(z)\pi_{\mathcal A}(z)\gamma\,dz,
    \end{equation}
    where the integral must be interpreted in the weak sense of vector-valued integration.
\end{theorem}

\begin{theorem}
    Let $\mathcal A\in\Sp(2d,\bR)$ have block decomposition \eqref{intro.blockA}, and consider the matrix $\mathcal A_\ast$ with block decomposition
    \begin{equation}
        \mathcal A_*=
        \begin{pmatrix}
            A_{12} & A_{11} & -A_{14} & -A_{13} \\
            A_{22} & A_{21} & -A_{24} & -A_{23} \\
            -A_{32} & -A_{31} & A_{34} & A_{33} \\
            -A_{42} & -A_{41} & A_{44} & A_{43}
        \end{pmatrix}.
    \end{equation}
    Then
    \begin{equation}
        \scal{\pi_{\mathcal A}(z)f}{g}=\scal{f}{\pi_{\mathcal A_\ast}(z)g}, \qquad f,g\in\cS(\rd), \, z\in\rdd.
    \end{equation}
    In particular, if $\pi_{\mathcal A}(z)$ extends to a continuous linear operator on $L^2(\rd)$, then
    \begin{equation}
        \pi_{\mathcal A}(z)^\ast=\pi_{\mathcal A_*}(z), \qquad z\in\rdd.
    \end{equation}
\end{theorem}

The rescaled-STFT formula also identifies the mapping properties of the metaplectic atoms.
\begin{theorem}
    Let $\mathcal A\in\Sp(2d,\bR)$ such that $W_{\mathcal A}$ is shift-invertible, and $z\in\rdd$.
    \begin{enumerate}[(i)]
        \item $\pi_{\mathcal A}(z)$ is a surjective scalar multiple of a unitary operator on $L^2(\rd)$, and
        \begin{equation}
            \|\pi_{\mathcal A}(z)f\|_2=|\det(E_{\mathcal A})|^{-1/2}\|f\|_2, \qquad f\in L^2(\rd).
        \end{equation}
        \item $\pi_{\mathcal A}(z)$ is a topological isomorphism on $\cS(\rd)$.
        \item $\pi_{\mathcal A}(z)$ is a topological isomorphism on $\cS'(\rd)$.
    \end{enumerate}
\end{theorem}

Sampling these atoms gives the corresponding frame system.
\begin{definition}
    Let $\mathcal A\in\Sp(2d,\bR)$ such that $\pi_{\mathcal A}(z)$ extends to a continuous linear operator on $L^2(\rd)$ for all $z\in\rdd$, $g\in L^2(\rd)$, and $\Lambda\subset\rdd$ a lattice. We define the metaplectic Gabor system as the set
    \begin{equation}
        \cG_{\mathcal A}(g,\Lambda)=\{\pi_{\mathcal A}(\lambda)g\}_{\lambda\in\Lambda}.
    \end{equation}
    We say that $\cG_{\mathcal A}(g,\Lambda)$ is a metaplectic Gabor frame if it is a frame for $L^2(\rd)$, i.e., if there exist constants $A,B>0$ such that
    \begin{equation}
        A\norm{f}_{2}^2\leq \sum_{\lambda\in\Lambda}|\scal{f}{\pi_{\mathcal A}(\lambda)g}|^2\leq B\norm{f}_{2}^2, \qquad f\in L^2(\rd),
    \end{equation}
    with metaplectic Gabor coefficients
    \begin{equation}
        c_\lambda=\scal{f}{\pi_{\mathcal A}(\lambda)g}=W_\cA(f,g)(\lambda).
    \end{equation}
    If $\cG_{\mathcal A}(g,\Lambda)$ is a metaplectic Gabor frame for $L^2(\rd)$, then $\|(\scal{f}{\pi_{\mathcal A}(\lambda)g})\|_{\ell^2}$ is an equivalent norm for $L^2(\rd)$. Explicitly, the frame condition can be written as
    \begin{equation}
        A\norm{f}_{2}^2\leq \sum_{\lambda\in\Lambda}|W_{\mathcal A}(f,g)(\lambda)|^2\leq B\norm{f}_{2}^2, \qquad f\in L^2(\rd).
    \end{equation}
\end{definition}

For shift-invertible representations this new frame notion is not detached from the classical one. The rescaled-STFT formula identifies it exactly with an ordinary Gabor frame after deforming the window and the lattice.

\begin{theorem}[Metaplectic frames as deformed Gabor frames]\label{thm:metaplectic-frame-equivalence}
Assume that $W_{\mathcal A}$ is shift-invertible and let $\widehat{\delta_{\mathcal A}}$ be the deformation operator in Theorem~\ref{thm:rescaled-stft}. Let $g\in L^2(\rd)$ and let $\Lambda\subset\rdd$ be discrete. Then $\cG_{\mathcal A}(g,\Lambda)$ is a metaplectic Gabor frame with bounds $A,B$ if and only if
\begin{equation}
 \cG(\widehat{\delta_{\mathcal A}}g,E_{\mathcal A}^{-1}\Lambda)
\end{equation}
is a Gabor frame with bounds $|\det(E_{\mathcal A})|A$ and $|\det(E_{\mathcal A})|B$.
\end{theorem}
Thus the metaplectic discretization deforms, rather than replaces, the Gabor discretization developed in Section~\ref{sec:gabor}.

For the remainder of this subsection, assume that $W_{\mathcal A}$ is shift-invertible. The preceding equivalence transports the classical frame theory to this setting, including the structure of the canonical dual.

Given a lattice $\Lambda\subset\rdd$ and $g\in\cS(\rd)$, we can consider the (metaplectic) analysis operator
\begin{equation}
C_{\mathcal A,g}f=(\langle f,\pi_{\mathcal A}(\lambda)g\rangle)_{\lambda\in\Lambda}, \quad \text{ for } f\in L^2(\rd),
\end{equation}
the (metaplectic) synthesis operator
\begin{equation}
D_{\mathcal A,g}c=\sum_{\lambda\in\Lambda}c_\lambda\pi_{\mathcal A}(\lambda)g, \quad \text{ for } c\in \ell^2(\Lambda),
\end{equation}
and the (metaplectic) Gabor frame operator
\begin{equation}
\mathcal S_{\mathcal A,g}f=D_{\mathcal A,g}C_{\mathcal A,g}f=\sum_{\lambda\in\Lambda}\langle f,\pi_{\mathcal A}(\lambda)g\rangle \pi_{\mathcal A}(\lambda)g.
\end{equation}

The canonical dual frame of $\cG_{\mathcal A}(g,\Lambda)$ is again a metaplectic Gabor frame $\cG_{\mathcal A}(\gamma_{\mathcal A},\Lambda)$. When $g\in\cS(\rd)$, the ordinary Schwartz-window duality result and the metaplectic deformation imply $\gamma_{\mathcal A}\in\cS(\rd)$ as well.

Consequently, if $\cG_{\mathcal A}(g,\Lambda)$ is a metaplectic Gabor frame for $L^2(\rd)$ with frame bounds $A$ and $B$, then we have the reconstruction formulas
\begin{equation}
    f=\sum_{\lambda\in\Lambda}\scal{f}{\pi_{\mathcal A}(\lambda)g}\pi_{\mathcal A}(\lambda)\gamma_{\mathcal A}, \quad f=\sum_{\lambda\in\Lambda}\scal{f}{\pi_{\mathcal A}(\lambda)\gamma_{\mathcal A}}\pi_{\mathcal A}(\lambda)g,
\end{equation}
and these series converge unconditionally in the norm of $L^2(\rd)$. The corresponding coefficient estimates are precisely the frame inequalities already built into the definition, while the canonical dual carries the reciprocal bounds. The coefficient characterizations below extend these Hilbert-space bounds to modulation norms.

Furthermore, the characterization of modulation spaces via Gabor frames yields an equivalent discrete norm for modulation spaces in terms of metaplectic Gabor frames.

\begin{theorem}[Discrete characterization in fixed phase-space coordinates]\label{thm:metaplectic-discrete}
Let $0<p,q\leq\infty$, $0\neq g\in\cS(\rd)$, and let $W_{\mathcal A}$ be shift-invertible. Let $m$ be $v$-moderate for a polynomially bounded control weight $v$, with $m\asymp m\circ E_{\mathcal A}^{-1}$. Assume that $\cG_{\mathcal A}(g,\Lambda)$ is a frame. If $p\neq q$, assume additionally that $E_{\mathcal A}$ is upper block triangular. Then the analysis and synthesis maps are bounded,
\begin{equation}
 C_{\mathcal A,g}:M_m^{p,q}\longrightarrow\ell_m^{p,q}(\Lambda),
 \qquad D_{\mathcal A,g}:\ell_m^{p,q}(\Lambda)\longrightarrow M_m^{p,q},
\end{equation}
where the sequence norm is the geometric norm \eqref{eq:geometric-sequence-norm}. Moreover,
\begin{equation}
 \|f\|_{M_m^{p,q}}
 \asymp\|C_{\mathcal A,g}f\|_{\ell_m^{p,q}(\Lambda)}
 \asymp\|C_{\mathcal A,\gamma_{\mathcal A}}f\|_{\ell_m^{p,q}(\Lambda)}.
\end{equation}
The reconstruction series converge unconditionally in $M_m^{p,q}$ when $p,q<\infty$, and weak-$\ast$ in $M_{1/v}^\infty$ at an infinite exponent. The norm-equivalence constants may depend on the exponents, weights, windows, and lattice.
\end{theorem}

\begin{proof}
Write $E=E_{\mathcal A}$ and $h=\widehat{\delta_{\mathcal A}}g$. By \eqref{eq:rescaled-stft},
\begin{equation}\label{eq:metaplectic-sampled-stft}
 |C_{\mathcal A,g}f(\lambda)|
 =|\det E|^{-1/2}|V_hf(E^{-1}\lambda)|.
\end{equation}
Theorem~\ref{thm:metaplectic-frame-equivalence} identifies the frame on the right with an ordinary Gabor frame on $E^{-1}\Lambda$. The weight compatibility and \eqref{eq:triangular-mixed-change} identify its geometric coefficient norm with the one on $\Lambda$; for $p=q$ an arbitrary invertible change suffices. The ordinary Gabor analysis, synthesis, and reconstruction statements therefore apply. The canonical dual is transported by the same deformation.
\end{proof}

For an arbitrary invertible $E_{\mathcal A}$, the full discrete characterization is retained by transporting the sequence norm together with the sampling coordinates. Explicitly, define
\begin{equation}\label{eq:transported-sequence-norm}
 \|c\|_{\ell_{m,E_{\mathcal A}}^{p,q}(\Lambda)}
 :=\|(c_{E_{\mathcal A}\mu})_{\mu\in E_{\mathcal A}^{-1}\Lambda}\|_{\ell_m^{p,q}(E_{\mathcal A}^{-1}\Lambda)}.
\end{equation}
Then \eqref{eq:metaplectic-sampled-stft} and ordinary Gabor theory give the same analysis, synthesis, and norm-equivalence statements with this transported norm, without a triangularity assumption. This is the deformed-coordinate interpretation of the discrete result in \cite{CorderoGiacchi2024MetaplecticGabor}. For $p\neq q$, it must be distinguished from keeping the original time-frequency order fixed on $\Lambda$.

\subsection{A recent refinement: exponential Gabor sparsity}

The preceding sections prepare three ingredients for a quantitative application: Gabor matrices encode an operator in phase space, Gelfand-Shilov windows provide controlled exponential localization, and the symplectic projection describes the underlying deformation. They meet in the study of sparse Gabor representations of metaplectic propagators.

For such windows, Cordero, Giacchi, Pucci and Trapasso obtain explicit exponential decay of the Gabor matrix away from the graph of the associated symplectic map \cite{CorderoGiacchiPucciTrapasso2026}. These estimates quantify dispersion and wave-packet spreading, and yield confinement results for quadratic Schr\"odinger propagators. The examples include the harmonic oscillator, a free particle in a constant magnetic field, and fractional Fourier transforms. Thus the continuous geometry and the discrete coefficient description do more than characterize a function space: together they describe how localized signals evolve.

\section{Concluding perspectives}
\label{sec:conclusion}

The history of modulation spaces reveals the stability of Feichtinger's local-global principle across several distinct kinds of development. At fixed exponents and weights, uniform decompositions, the STFT, and Gabor coefficients provide equivalent descriptions of the same regularity. The quasi-Banach and ultradistribution extensions change the functional-analytic setting as well, but retain the underlying organization of phase-space information. Keeping the full weight $m(x,\xi)$ visible makes these connections precise: it records which variables carry decay or regularity and how this information transforms.

The metaplectic results make the analyzer aspect especially explicit. Shift-invertible $W_{\mathcal A}$ are rescaled STFTs up to a chirp and a deformation of the window. Under the appropriate weight and mixed-norm compatibility conditions, they therefore provide alternative measurements of Feichtinger's spaces rather than a competing scale. Upper- and lower-triangular conditions express compatibility with the integration order for modulation and Wiener amalgam norms. After sampling, the same distinction persists: one must either preserve the original order of the coefficient norm or transport that norm together with the lattice.

From Feichtinger's 1983 LCA construction to the recent metaplectic and sparse-Gabor developments, the recurring principle is the same: localize phase-space information and then measure its global organization in a flexible norm. Its continuing ability to accommodate new analyzers, weights, generalized functions, operator classes, and evolution equations is a fitting mathematical tribute to Hans Feichtinger and to the breadth of the framework he initiated.
\section*{Acknowledgements}
Antonio Caputo, Elena Cordero and Gianluca Giacchi have been supported by the Gruppo Nazionale per l’Analisi Matematica, la Probabilità e le loro Applicazioni (GNAMPA) of the Istituto Nazionale di Alta Matematica (INdAM). Gianluca Giacchi has been funded by the Swiss National Science Foundation starting grant ``Multiresolution methods for unstructured data'' (TMSGI2\_211684).

\bibliographystyle{abbrv}

\end{document}